\documentclass{amsart}
\usepackage{amssymb, amsmath, verbatim}
\usepackage[all]{xy}

\theoremstyle{plain}
\newtheorem{lem}{Lemma}[section]
\newtheorem{thm}[lem]{Theorem}
\newtheorem{cor}[lem]{Corollary}

\theoremstyle{definition}
\newtheorem{defn}[lem]{Definition}
\newtheorem{example}[lem]{Example}

\theoremstyle{remark}

\newtheorem{question}{Question}[section]

\numberwithin{equation}{section}

\def\repeattheoremhelper#1#2#3#4#5{
  \theoremstyle{#1}
  \newtheorem*{#4}{#2 \ref{#3}}
  \begin{#4}
    #5
  \end{#4}
}

\def\repeattheorem#1#2{
  \repeattheoremhelper{plain}{#1}{#2}{repeat#2}{\csname state#2\endcsname}
}

\newcommand{\mc}[1]{\mathcal{#1}}
\newcommand{\mf}[1]{\mathfrak{#1}}

\newcommand{\eg}{{\it e.g.},}

\newcommand{\card}[1]{\lvert #1\rvert}
\newcommand{\al}{\aleph}
\newcommand{\af}{\alpha}
\newcommand{\bt}{\beta}
\newcommand{\gm}{\gamma}
\newcommand{\dl}{\delta}
\newcommand{\om}{\omega}
\newcommand{\ka}{\kappa}
\newcommand{\lm}{\lambda}
\newcommand{\oml}{\om_1}
\newcommand{\alo}{\al_0}
\newcommand{\eps}{\varepsilon}
\newcommand{\lapps}[1]{long $#1$-approximation sequence}

\newcommand{\lc}{\lm}
\newcommand{\lappsl}{\lapps{\lc}}

\newcommand{\lappso}{\lapps{\oml}}

\newcommand{\da}[1][]{\daleth_{#1}(\af)}
\newcommand{\bighcard}{\theta}
\newcommand{\bigh}{H(\bighcard)}
\newcommand{\bigwo}{\sqsubset_\bighcard}
\newcommand{\bighs}{\mf{H}}
\newcommand{\elemsub}{\prec}
\newcommand{\subh}{\elemsub\bighs}

\newcommand{\restrict}{\upharpoonright}
\newcommand{\inv}[1]{#1^{-1}}

\newcommand{\subreg}{\leq_{\mathrm{reg}}}

\newcommand{\trile}{\triangleleft}

\newcommand{\vn}{\varnothing}

\newcommand{\istst}{it suffices to show that}
\newcommand{\Istst}{It suffices to show that}
\newcommand{\istf}{it suffices to find}

\newcommand{\wma}{we may assume}

\newcommand{\wrt}{with respect to}

\newcommand{\suchthat}{\mathrel{}|\mathrel{}}
\newcommand{\Suchthat}{\mathrel{}\middle|\mathrel{}}

\newcommand{\dksfn}[2]{{$#2$}\nobreakdash-$\mathrm{SFN}_{#1}$}
\newcommand{\dkfn}[2]{{$#2$}\nobreakdash-$\mathrm{FN}_{<#1}$}
\DeclareMathOperator{\cf}{cf}
\DeclareMathOperator{\supp}{supp}
\DeclareMathOperator{\op}{op}
\DeclareMathOperator{\dom}{dom}
\DeclareMathOperator{\id}{id}
\DeclareMathOperator{\ran}{ran}
\DeclareMathOperator{\powset}{\mc{P}}

\DeclareMathOperator{\ult}{Ult}

\DeclareMathOperator{\clop}{Clop}
\newcommand{\tfae}{the following are equivalent}

\newcommand{\gen}[2][]{\left\langle #2\right\rangle_{#1}}
\newcommand{\overlap}{\gen{A_i\cap\bigcup_{j\not=i}A_j}}
\newcommand{\scepin}{\v S\v cepin}
\newcommand{\proves}{\vdash}

\DeclareMathOperator{\spowop}{SP}

\newcommand{\homeo}{\cong}
\newcommand{\imp}{\Rightarrow}

\newcommand{\weight}[1]{w\left(#1\right)}

\DeclareMathOperator{\Exp}{Exp}
\newcommand{\mb}{\mathbb}
\DeclareMathOperator{\interior}{int}
\DeclareMathOperator{\cl}{cl}

\DeclareFontFamily{U}{mathb}{\hyphenchar\font45}
\DeclareFontShape{U}{mathb}{m}{n}{
      <5> <6> <7> <8> <9> <10> gen * mathb
      <10.95> mathb10 <12> <14.4> <17.28> <20.74> <24.88> mathb12
      }{}
\DeclareSymbolFont{mathb}{U}{mathb}{m}{n}
\DeclareFontSubstitution{U}{mathb}{m}{n}
\DeclareFontFamily{U}{mathx}{\hyphenchar\font45}
\DeclareFontShape{U}{mathx}{m}{n}{
      <5> <6> <7> <8> <9> <10>
      <10.95> <12> <14.4> <17.28> <20.74> <24.88>
      mathx10
      }{}
\DeclareSymbolFont{mathx}{U}{mathx}{m}{n}
\DeclareFontSubstitution{U}{mathx}{m}{n}

\DeclareMathSymbol{\bigboxplus}{1}{mathx}{"D0}
\DeclareMathSymbol{\bigboxtimes}{1}{mathx}{"D2}
\DeclareMathAccent{\widecheck}{0}{mathx}{"71}

\newcommand{\bigcomm}{
  \mathop{
    \vcenter{
      \hbox{\oalign{\noalign{\kern-.3ex}\hfil$\vert$\hfil\cr
        \noalign{\kern-.7ex}
        $\smile$\cr\noalign{\kern-.3ex}}}
    }
  }  
}

\begin{document}

\title{Team games, hypergraph spaces, and projective Boolean algebras}
\dedicatory{In memory of Ken Kunen, incisive scholar and patient teacher.}
\author{David Milovich}
\email{david.milovich@welkinsciences.com}
\urladdr{http://dkmj.org}
\address{
  Welkin Sciences\\
  2 N Nevada Ave, Suite 1280\\
  Colorado Springs, CO 80903, USA
}
\date{Submitted June 30, 2021; revised Feb. 24, 2022}
\begin{abstract}
  We modify the game Fuchino, Koppelberg, and Shelah used
  to characterize the $\kappa$-Freese-Nation property
  for a given Boolean algebra $A$, replacing players I and II
  each with a team of $n$ players with limited information. 
  We show that $A$ is tightly $\kappa$-filtered exactly when team II has
  a winning strategy for every finite team size.
  
  Case $\kappa=\aleph_0$ characterizes projective Boolean algebras
  and, hence, Dugundji spaces.
  In terms of the open-open game of Daniels, Kunen, and Zhou,
  this characterization is a team version of ``very I-favorable.''
  We similarly characterize Cohen algebras
  in terms of a team version of I-favorability.
  
  If $A$ is the clopen algebra of the space of $n$-uniform hypergraphs
  on $\kappa^{+n}$ that avoid copies of $[n+1]^n$, then
  team II has a winning strategy for our modified FKS game
  for team size $n-1$ but not $n$.
  For $n\geq 3$, this algebra also answers a question of Geschke
  when combined with a locally ${<}\kappa$-sized characterization
  of tightly $\kappa$-filtered Boolean algebras that we prove.
  Case $\kappa=\aleph_0$ includes a locally finite characterization
  of projective Boolean algebras.
\end{abstract}


\maketitle

\subsection*{Outline}
In Section~\ref{secteamclubgame}, we define uncoordinated winning strategies
for team games and state a theorem that uses such strategies to characterize
projective Boolean algebras and their higher-cardinal generalization,
Geschke's tightly $\ka$-filtered Boolean algebras.
We also state a theorem saying that a natural family of hypergraph spaces
witness that the preceding theorem is optimal \wrt\ team size.
These two theorems also include statements that answer a question
of Geschke about tightly $\ka$-filtered Boolean algebras.

In Section~\ref{secprojteam}, we use elementary submodel techniques,
including Davies trees, to prove (strengthenings of) half of each of
the above two theorems.

In Section~\ref{secopenopen}, we introduce a team version of the
open-open game of Daniels, Kunen, and Zhou and use the techniques
of Section~\ref{secprojteam} to characterize Cohen algebras in terms
of uncoordinated winning strategies. We also explain the connections
between this paper's results and Daniels, Kunen, and Zhou's concepts of
I-favorable and very I-favorable spaces.

In Section~\ref{seclocfin}, we complete the proofs of the two theorems
stated in Section~\ref{secteamclubgame}. As a corollary, we characterize
Dugundji spaces in terms of inverse limits of maps between finite discrete
spaces. These results are proved using higher-arity forms of
conditional independence between tuples of Boolean subalgebras.
We also obtain some purely finitary consequences related to these
conditional independence concepts, including a higher-arity version of
the Craig interpolation property for propositional logic.

We conclude with some open questions in Section~\ref{secquestions}.

\section{Teams, clubs, and Boolean algebras}\label{secteamclubgame}

\subsection{Uncoordinated teams}
A two-player game where two players alternately play $n$-tuples can be 
interpreted as a game between two teams of $n$ players each.
Under this interpretation, it makes sense to ask if a team $T$ has
an \emph{uncoordinated} winning strategy in which the $m$th
player of $T$ knows the previous moves of the $m$th player of
the opposing team but is ignorant of all other players' moves.
(Note that an uncoordinated winning strategy must win even
against a coordinated opposing team.)
For a quick but nontrivial example,
consider the following variant of the Banach-Mazur game.

\begin{example}
Fix a subset $E$ of a product $X=\prod_{m=1}^nX_m$ of topological spaces.
In round $k$, the $m$th player of I plays a nonempty open $U_{m,k}\subset X_m$
and then the $m$th player of II plays a nonempty open $V_{m,k}\subset X_m$.
These opens are required to be nested: $U_{m,k}\supset V_{m,k}\supset U_{m,k+1}$.
After $\om$ rounds, team II wins iff $\bigcap_{k<\om}\prod_{m=1}^nV_{m,k}\subset E$.
Then II has a winning strategy iff $E$ is comeager in $X$.
In contrast, II has an uncoordinated winning strategy iff $E$ 
contains a product $\prod_{m=1}^n B_m$ of comeager sets.
In particular, if $n\geq 2$, $X=\mb{R}^n$, 
and $\vec x\in E$ iff $x_1+\cdots+x_n$ is irrational,
then II has a winning strategy but no uncoordinated winning strategy
because $Y+Z=\mb{R}$ for all comeager $Y,Z\subset\mb{R}$.
\end{example}

This work's applications of uncoordinated strategies 
include new characterizations of projective Boolean algebras and 
of Geschke's generalization of these, tightly $\ka$-filtered Boolean algebras.
In particular, we will define a team game $\mc{G}(A,\ka,\tau)$ of length $\ka$
for arbitrary Boolean algebras $A$ and team sizes $\tau$
such that $A$ is tightly $\ka$-filtered iff II has an uncoordinated
winning strategy for all finite $\tau$ satisfying $\ka^{+\tau}\leq\card{A}$.
Moreover, for each $n\in[1,\om)$, there a Boolean algebra $A$ of
size $\ka^{+(n+1)}$ such that II has an uncoordinated winning strategy
for $\mc{G}(A,\ka,n)$ but not for $\mc{G}(A,\ka,n+1)$.

\begin{defn}\ 
\begin{itemize}
\item Given a cardinal $\tau$ and an ordinal $\eta$, in
a \emph{team game} of size $\tau$ and length $\eta$
is a two-player game with $\eta$ rounds such that in each round
player I must play a $\tau$-sequence and then player II must play a $\tau$-sequence.
We will call I and II \emph{teams} instead of players.
\item For each $\alpha<\tau$, $\beta<\eta$, and $T\in\{\mathrm{I},\mathrm{II}\}$, 
by \emph{play $\beta$ of player $\alpha$ of $T$} we mean
the $\alpha$th coordinate of the $\beta$th play of $T$.
\item For each $T\in\{\mathrm{I},\mathrm{II}\}$, a strategy $\sigma$ for $T$
is \emph{uncoordinated} if, for all $\alpha<\tau$ and all partial play histories 
$h,h'\in\dom(\sigma)$, if $\dom(h)=\dom(h')$ and 
$h(\beta)(\alpha)=h'(\beta)(\alpha)$ for all $\beta\in\dom(h)$, then
$\sigma(h)(\alpha)=\sigma(h')(\alpha)$.
\end{itemize}
\end{defn}

\subsection{Projective Boolean algebras and the Freese-Nation property}
\begin{defn}
In a given category:
\begin{itemize}
\item An object $P$ is \emph{projective} iff
for every morphism $m\colon P\to Q$ and epimorphism $q\colon O\to Q$
we may factor $q$ as $m\circ p$ for some $p\colon O\to P$.
\item A \emph{retraction} is a morphism $r$ that has a right inverse $e$.
An object $R$ is a \emph{retract} of an object $O$ iff there is a retraction from $O$ to $R$.
\end{itemize}
\end{defn}
$$\xymatrix{ & P\ar[d]^m &&&& R\\ 
O\ar@{-->}[ru]^p\ar@{->>}[r]_q & Q &&& R\ar[r]_e\ar[ru]^\id&O\ar[u]_r}$$ 

In various concrete categories from algebra, including 
the category Bool of Boolean algebras and Boolean homomorphisms,
an algebra $P$ is projective iff every homomorphism $h$
from $P$ to a quotient algebra $A/I$ is of the form $g/I$
for some homomorphism $g\colon P\to A$. Moreover, for
Bool and other categories of algebras with enough free algebras,
$P$ is projective iff it is the retract of a free algebra.
In particular, a free Boolean algebra is projective and 
a projective Boolean algebra is ccc because it is a retract and, 
hence, subalgebra of some free Boolean algebra.
For additional background information about projective
Boolean algebras, we recommend \cite[Ch. 1]{hs} or \cite{koppproj}.

In topology, the \emph{Dugundji spaces} (a.k.a., \emph{absolute extensors of dimension zero}
or AE(0) spaces) are exactly the Stone spaces\footnote{
  That is, zero-dimensional compact Hausdorff spaces, also known as Boolean spaces.}
with projective clopen algebras and, hence, exactly the retracts of powers of 2.
In particular, all countable Boolean algebras are projective because
every closed subset of the Cantor set is easily seen to be a retract of the Cantor set.

The eponymous \emph{Freese-Nation property} (FN) was introduced
as part of a characterization of projective lattices.\,\cite{freesenation}

\begin{defn}
A poset $P$ has the \emph{FN} iff there is a map $f\colon P\to[P]^{<\alo}$
such that for every nonempty interval $[x,y]$, the set
$f(x)\cap f(y)\cap [x,y]$ is also nonempty. 
\end{defn}

In words, the Freese-Nation property demands that we associate
to every point in a poset a finite cloud such that the finite clouds
of any two comparable points intersect at some interpolant.
For additional background about the Freese-Nation property, we recommend \cite[Chs. 2-3]{hs}.

It is easily shown that free Boolean algebras have the FN
and that retracts preserve the FN. So, projectivity implies the FN.
However, the converse is false and interestingly so, especially from
the topological point of view. \scepin\ isolated the class of \emph{openly generated} 
(a.k.a., $\ka$-metrizable\footnote{Here, $\ka$ is just a letter, not a parameter.})
compact Hausdorff spaces and proved many interesting results about this class
and the class of Dugundji spaces.\,\cite{sckmet,scfunct,sctoplim}
In their book \cite{hs}, Heindorf and Shapiro observed that for Stone spaces
open generation is exactly having a clopen algebra with the FN.
Heindorf and Shapiro's book also presented the zero-dimensional case of
several of \scepin's results in algebraic terms. In particular, we note the following results,
stating them in a mix of topological and algebraic language.
\begin{itemize}
\item A Boolean algebra is free iff it projective and all its ultrafilters have equal character.
\item A Boolean algebra of size at most $\al_1$ is projective iff it has the FN.
\item Any Boolean algebra with the FN is ccc.
\item The Vietoris hyperspace operator $\Exp$ and symmetric $n$th power\footnote{
The symmetric $n$th power of a space $X$ is the space of multisets of points of $X$ 
with cardinality $n$. In other words, it is the quotient of the product space $X^n$
induced by identifying a tuple $f\colon n\to X$ with $f\circ\sigma$ for all permutations $\sigma$ of $n$.}
operator $\spowop^n$ preserve open generation.\,\cite{sckmet}
\item For all $\lm\geq\om_2$ and $n\geq 2$, the clopen algebras of 
$\Exp(2^\lm)$ and $\spowop^n(2^\lm)$ have the FN but are not projective.
(Moreover, Shapiro \cite{shapiro} showed that $\Exp(2^\lm)$ is not even a continuous image of
a power of 2.)
\end{itemize}

One striking consequence of the above results is that 
$\spowop^n\left(2^{\oml}\right)$ is homeomorphic to $2^{\oml}$ for all $n$ but
$\spowop^n\left(2^{\om_2}\right)$ is not homeomorphic to $2^{\om_2}$ for $n\geq 2$.
\scepin\ showed that moreover
$\spowop^n\left(2^{\om_2}\right)\not\homeo\spowop^m\left(2^{\om_2}\right)$
for $n>m\geq 1$. Part of my motivation for this work was to find an analogous
symmetry breaking at weight $\al_3$ or higher. Theorems~\ref{hypergraphsame}
and \ref{hypergraphdifference} are the kind of thing I was looking for.

\subsection{The FKS game}
Fuchino, Koppelberg, and Shelah \cite{fks} generalized the FN
by adding a cardinal parameter $\ka$.
For regular $\ka$, they characterized the $\ka$-FN in terms of a game of length $\ka$.
Though their characterization is for all posets, we will only state it for Boolean algebras.

\begin{defn}
Given an infinite cardinal $\ka$,
a poset $P$ has the \emph{$\ka$-FN} iff there is a map $f\colon P\to[P]^{<\ka}$
such that for every nonempty interval $[x,y]$, the set
$f(x)\cap f(y)\cap [x,y]$ is also nonempty. 
\end{defn}

The $\al_1$-FN, also known as the \emph{weak Freese-Nation property} or \emph{WFN}
is of particular interest to set-theorists because many of the properties of the Cohen model
follow from the fact that $\powset(\om)$ has the WFN in that model.\,\cite{fks}

We now recall some essential results about the $\ka$-FN.

\begin{defn}
Given an infinite cardinal $\ka$, a set $S$, and $\mc{E}\subset[S]^{\leq\ka}$:
\begin{itemize}
\item $\mc{E}$ is \emph{closed} if it closed \wrt\
unions of ascending sequences of length at most $\ka$.
\item $\mc{E}$ is a \emph{club} if it closed and cofinal in $[S]^{\leq\ka}$
(ordered by inclusion).
\item In the \emph{FKS game for $\mc{E}$},
players I and II respectively choose $X_\af,Y_\af\in[S]^{<\ka}$
such that $X_\af\subset Y_\af\subset X_\bt$ for all $\af<\bt<\ka$.
II wins iff $\bigcup\vec{Y}\in\mc{E}$.
\end{itemize}
\end{defn}

\begin{lem}
II has a winning strategy for the FKS game for $\mc{E}$ iff
$\mc{E}$ contains a club.
\end{lem}
\begin{proof}
See the proof of Lemma~\ref{geschketogame}.
\end{proof}

\begin{defn}
Given an infinite cardinal $\ka$ and Boolean algebras $A,B$:
\begin{itemize}
\item We say an injective Boolean homomorphism $e\colon A\to B$ 
is a \emph{$\ka$-embedding} and write $e\colon A\leq_\ka B$
iff for every $y\in B$ the set $\{x\in A\suchthat e(x)\leq y\}$
has cofinality less than $\ka$.
\item If $A$ is a subalgebra of $B$, then 
we say $A$ is a \emph{$\ka$-subalgebra} of $B$ and write $A\leq_\ka B$
iff $\id_A\colon A\leq_\ka B$ where $\id_A=(x\mapsto x)_{x\in A}$.
\end{itemize}
\end{defn}

Thus, $A\leq_\ka B$ iff every ideal of the form $\{x\in A\suchthat x\leq_B y\}$
has a cofinal subset of size less than $\ka$. Therefore, $A\leq_{\alo} B$ iff
every set of the form $\{x\in A\suchthat x\leq_B y\}$ has a maximum element.\footnote{
The special case $A\leq_{\alo}B$ is more commonly denoted by
  $A\leq_{\text{rc}}B$. (See, \eg\ \cite{hs}.) The ``rc'' stands for
  \emph{relatively complete}.}

\begin{thm}[\cite{fks}]
Given a regular infinite cardinal $\ka$, a Boolean algebra $B$ has the $\ka$-FN 
iff II has a winning strategy in the FKS game for
$\{A\in[B]^{\leq\ka}\suchthat A\leq_\ka B\}$.
\end{thm}

Also in \cite{fks}, many facts about the FN are generalized to the $\ka$-FN. In particular:

\begin{thm}[\cite{fks}]
Given a regular infinite cardinal $\ka$ and a Boolean $A$:
\begin{itemize}
\item If $\card{A}\leq\ka$, then $A$ has the $\ka$-FN.
\item If $A$ has the $\ka$-FN, then $A$ has the $\ka$-cc.
\item If $A$ has the $\ka$-FN, then so do all retracts of $A$.
\end{itemize}
\end{thm}

\subsection{FKS games for teams}

\begin{defn}\label{deftight}
Given a Boolean algebra $A$:
\begin{itemize}
\item For each $S\subset A$, $\gen{S}$ is the Boolean subalgebra generated by $S$.
\item (Geschke \cite{geschke}) 
Given a regular infinite cardinal $\ka$, $A$ is \emph{tightly $\ka$-filtered}
iff there is a transfinite sequence $(x_\af)_{\af<\eta}$ such that 
$A=\gen{\{x_\af\suchthat\af<\eta\}}$ and, for all $\af<\eta$,
$$\gen{\{x_\bt\suchthat\bt<\af\}}\leq_\ka\gen{\{x_\bt\suchthat\bt<\af+1\}}.$$
\end{itemize}
\end{defn}

Geschke's motivation for the above definition is the following result of Koppelberg.

\begin{lem}[{\cite[2.8]{koppproj}}]
A Boolean algebra is projective iff it is tightly $\alo$-filtered.
\end{lem}

Geschke generalized several results of \scepin. In particular:

\begin{thm}[\cite{geschke}]
Given a regular infinite cardinal $\ka$ and a Boolean $A$:
\begin{itemize}
\item If $A$ is tightly $\ka$-filtered, then $A$ has the $\ka$-FN.
\item If $A$ has the $\ka$-FN and has size at most $\ka^+$, then $A$ is tightly $\ka$-filtered.
\item If $A$ is tightly $\ka$-filtered, then so are all retracts of $A$.
\item Given an infinite cardinal $\lm$, the clopen algebra of the symmetric square
$\spowop^2(2^\lm)$ is tightly $\ka$-filtered iff $\lm\leq\ka^+$.
\end{itemize}
\end{thm}

To prove that $\clop\left(\spowop^2\left(2^{\ka^{++}}\right)\right)$
is not tightly $\ka$-filtered, Geschke used the following lemma.

\begin{defn}\label{gmapdef}
Given a Boolean algebras $A$, a regular infinite cardinal $\ka$, and 
a nonzero cardinal $\tau$, a \emph{$(\ka,\tau)$-Geschke map} for $A$ is a map 
$f\colon A\to[A]^{<\ka}$ such that $\gen{\bigcup\vec B}\leq_\ka A$
for all $\tau$-sequences $\vec B$ of $f$-closed subalgebras of $A$.
\end{defn}

Observe that if $\ka\leq\ka'$ and $\tau\geq\tau'$, then any
$(\ka,\tau)$-Geschke map is also a $(\ka',\tau')$-Geschke map.

\begin{lem}{\cite[Cor. 2.7]{geschke}}\label{geschkeunion}
  If $A$ is tightly $\ka$-filtered, then it has a $(\ka,2)$-Geschke map.
\end{lem}
\begin{proof}
  The essential step of Geschke's proof is to state that,
  by his Theorems 2.4 and 2.5, every tightly $\ka$-filtered $A$
  is the union of some $\mc{C}\subset[A]^{<\ka}$ of such that
  $\gen{\bigcup\mc{B}}\leq_\ka A$ for every $\mc{B}\subset\mc{C}$.
  For uncountable $\ka$, his Theorem 2.5(vii) provides exactly such a family
  $\mc{C}$. But for $\ka=\alo$, there is a mistake:
  the family $\mc{C}$ provided by his Theorem 2.4(vi)
  has elements that are merely countable, not finite. Fortunately,
  Geschke's proof can be repaired by observing that his proof of his Theorem~2.5
  actually works for $\ka=\alo$ too because if $\dl$ is an ordinal,
  $f\colon\dl\to[\dl]^{<\alo}$, and $f(\af)\subset\af$ for all $\af<\dl$,
  then every finite $s\subset\dl$ has a finite $f$-closure,
  a fact easily proved by induction on $\max(s)$.
  For a more self-contained argument, see the proofs of Lemmas \ref{sfntogeschke}
  and \ref{tighttosfn}.
\end{proof}
Geschke then raised the question of whether the converse of this lemma is true.
\begin{quote}
  It would be interesting to know whether the existence of a function as
  in Corollary 2.7 already characterizes tight $\ka$-filteredness.\footnote{
    Actually, Geschke required of the function $f$ in Definition~\ref{gmapdef}
    the stronger property that $\gen{\bigcup\vec B}\leq_\ka A$
    for all $\tau$-sequences $\vec B$ of $f$-closed \emph{subsets} of $A$.
    I interpret Geschke to intend \emph{subalgebras}, not \emph{subsets},
    for three reasons.
    First, Geschke's definition is strange if interpreted literally
    because a $(\ka,1)$-Geschke map $f$ would also be a $(\ka,\tau)$-Geschke map
    for all $\tau$ because any union of $f$-closed sets is itself an $f$-closed set.
    Second, such an $f$ immediately implies tight $\ka$-filteredness
    by \cite[Thm. 2.4(vi)]{geschke} and \cite[Thm. 2.5(viii)]{geschke},
    giving Geschke's question a trivial answer.
    Third, whenever Geschke uses a $(\ka,2)$-Geschke map $f$,
    the relevant $f$-closed subsets are also subalgebras.}
\end{quote}

We will eventually negatively answer Geschke's question by showing that
there is a Boolean algebra of size $\ka^{+3}$ that has something better
than a $(\ka,2)$-Geschke map
yet is not $\ka$-tightly filtered because there is no uncoordinated
winning strategy for a certain game with team size 3.
On the other hand, we also will show that having $(\ka,n)$-Geschke maps
for all $0<n<\om$ does characterize tight $\ka$-filteredness.

\begin{defn}
  Given a Boolean algebra $A$, an infinite regular cardinal $\ka$,
  and a nonzero team size $\tau$,
  define the \emph{FKS team game} $\mc{G}(A,\ka,\tau)$ as follows.
  There are $\ka$ rounds. In round $\af$, team I plays $(X_{\mu,\af})_{\mu<\tau}$
  and then II plays $(Y_{\mu,\af})_{\mu<\tau}$
  such that $X_{\mu,\af},Y_{\mu,\af}\in[A]^{<\ka}$
  and $\bigcup_{\bt<\af}Y_{\mu,\bt}\subset X_{\mu,\af}\subset Y_{\mu,\af}$
  for all $\mu<\tau$.
  II wins iff $\gen{\bigcup_{\mu<\tau}\bigcup_{\af<\ka}Y_{\mu,\af}}\leq_\ka A$.
\end{defn}

\begin{lem}\label{geschketogame}
Given $(A,\ka,\tau)$ as above, we have
\eqref{taukappafunc}$\Rightarrow$\eqref{taukappaclub}$\Rightarrow$\eqref{taukappagame}
for the following statements. If $\tau\leq\ka$, then we also have 
\eqref{taukappagame}$\Rightarrow$\eqref{taukappaclub}.
\begin{enumerate}
\item\label{taukappafunc} There is a $(\ka,\tau)$-Geschke map $f$ for $A$.
\item\label{taukappaclub} There is a club $\mc{E}\subset[A]^{\leq\ka}$ 
such that $\gen{\bigcup\mc{B}}\leq_\ka A$ for all $\mc{B}\in[\mc{E}]^{\leq\tau}$.
\item\label{taukappagame} II has an uncoordinated winning strategy $\sigma$
  for $\mc{G}(A,\ka,\tau)$.
\end{enumerate}
\end{lem}
\begin{proof}\ 
\begin{itemize}
\item[\eqref{taukappafunc}$\Rightarrow$\eqref{taukappaclub}]
Let $\mc{E}$ be the set of all $\ka$-sized $f$-closed subalgebras of $A$.
\item[\eqref{taukappaclub}$\Rightarrow$\eqref{taukappagame}] 
For each $S\in[A]^{\leq\ka}$, fix $E(S),g(S)$ such that $S\subset E(S)\in\mc{E}$ 
and $g(S)$ maps $\ka$ onto $E(S)$. Then player $\mu$ of II can ensure
$\bigcup_{\af<\ka}Y_{\mu,\af}\in\mc{E}$ by
playing according to the following recursion.
\begin{align*}
Y_{\mu,\af}&=X_{\mu,\af}\cup\{g(Z_{\mu,\bt})(\gm)\suchthat\bt,\gm<\af\}\\
Z_{\mu,\af}&=E\left(X_{\mu,\af}\cup\bigcup_{\bt<\af}Z_{\mu,\bt}\right)
\end{align*}
\item[\eqref{taukappagame}$\Rightarrow$\eqref{taukappaclub}]
Let $\mc{D}=\bigcap_{\mu<\tau}\mc{D}_\mu$ where $\mc{D}_\mu$
is the set of all unions $\bigcup_{\af<\ka}Y_{\mu,\af}$ obtainable
by playing according to $\sigma$. Each $\mc{D}_\mu$ contains a club.
Therefore, $\mc{D}$ contains a club if $\tau\leq\ka$.\qedhere
\end{itemize}
\end{proof}

We will eventually prove the next two theorems
and thus settle Geschke's question.

\begin{thm}[Main Theorem]\label{mainthm}
Given a regular infinite cardinal $\ka$ and a Boolean algebra $A$, \tfae.
\begin{enumerate}
\item $A$ is tightly $\ka$-filtered.\label{tightkamainthm}
\item For every nonzero cardinal $\tau$, $A$ has a $(\ka,\tau)$-Geschke map.
\item For every finite nonzero ordinal $\tau$ satisfying $\ka^{+\tau}\leq\card{A}$, 
$A$ has a $(\ka,\tau)$-Geschke map.
\item For every nonzero cardinal $\tau$, 
II has an uncoordinated winning strategy for $\mc{G}(A,\ka,\tau)$.\label{addkaclub}
\item For every finite nonzero ordinal $\tau$ satisfying $\ka^{+\tau}\leq\card{A}$, 
II has an uncoordinated winning strategy for $\mc{G}(A,\ka,\tau)$.
\end{enumerate}
\end{thm}

\begin{cor}\label{mainprojcor}
Given a Boolean algebra $A$, \tfae.
\begin{enumerate}
\item $A$ is projective\label{projmainprojcor}
\item For every nonzero cardinal $\tau$, $A$ has an $(\alo,\tau)$-Geschke map.
\item For every finite nonzero ordinal $\tau$ satisfying $\al_\tau\leq\card{A}$, 
$A$ has an $(\alo,\tau)$-Geschke map.
\item For every nonzero cardinal $\tau$, 
II has an uncoordinated winning strategy for $\mc{G}(A,\alo,\tau)$.\label{addclub}
\item For every finite nonzero ordinal $\tau$ satisfying $\al_\tau\leq\card{A}$, 
II has an uncoordinated winning strategy for $\mc{G}(A,\alo,\tau)$.
\end{enumerate}
\end{cor}

\begin{thm}[Main Example]
Given a regular infinite cardinal $\ka$ and $0<d<\om$,
let $C_{\ka,d}$ be the clopen algebra of the space of $d$-uniform hypergraphs
on $\ka^{+d}$ that avoid copies of $[d+1]^d$.
Then $C_{\ka,d}$ has an $(\alo,d-1)$-Geschke map but II does not have 
an uncoordinated winning strategy for $\mc{G}(C_{\ka,d},\ka,d)$.
\end{thm}

\begin{cor}
$C_{\ka,3}$ has a $(\ka,2)$-Geschke map but is not $\ka$-tightly filtered.
\end{cor}


\section{Proving team-game characterizations of tight $\ka$-filteredness}
\label{secprojteam}
In this section, we use team games and elementary submodel
techniques to prove half of the main theorem and half of the main example.

\subsection{Elementary submodel techniques}
\begin{defn}
  Given $\bighcard$ a regular uncountable cardinal,
  let $\bighs$ denote the first-order structure
  $(\bigh,\in,\bigwo)$ where $\bigh$ is the
  set of all sets with transitive closure smaller than $\bighcard$
  and $\bigwo$ is a well-ordering of $\bigh$.
  Given a set $M$, by $M\subh$ we mean that
  $M$ is the universe of an elementary substructure
  of $\bighs$.\footnote{
    That is, for each first-order sentence $\varphi$
    using non-logical symbols $\in$, $\bigwo$, and
    parameters from $M$, the sentence $\varphi$ is
    true in $(M,{\in}\cap M^2,{\bigwo}\cap M^2)$ iff it is true in $\bighs$.}
  Our convention is that $\bighcard$ is always implicitly chosen
  to be a regular uncountable cardinal sufficiently large for the argument at hand.
\end{defn}

We include a well-ordering in $\bighs$ so that $M,N\subh$ implies $M\cap N\subh$.

We will need the following generalization of Davies trees \cite{davies}
introduced by the author in \cite{msfn}.

\begin{defn}
  Given an ordinal $\eta$ and a regular uncountable cardinal $\lm$,
  a transfinite sequence $(M_\af)_{\af<\eta}$ is called
  a \emph{\lappsl} iff
  \begin{enumerate}
  \item $\card{M_\af}<\lm$,
  \item $[\bigh]^{<\lm}\cap M_\af\subset[M_\af]^{<\lm}$,
  \item $M_\af\subh$, and
  \item $(M_\bt)_{\bt<\af}\in M_\af$.
  \end{enumerate}
\end{defn}
Standard closing-off arguments show that every $x\in\bigh$
is an element of some $M_0$ satisfying (1), (2), and (3).
Therefore, we can extend any \lappsl\ $x=\vec M\in\bigh$ to a longer \lappsl.
Hence, for any $x\in\bigh$ and $\eta\leq\bighcard$, there is
a \lappsl\ $(M_\af)_{\af<\eta}$ with $x\in M_0$.

\begin{lem}[{\cite[Lem. 2.2]{msfn}}]
  Given a \lappsl\ $\vec M$, \tfae.
  \begin{itemize}
  \item $M_\bt\subset M_\af$
  \item $M_\bt\in M_\af\cup\{M_\af\}$
  \item $\bt\in(\af+1)\cap M_\af$
  \end{itemize}
\end{lem}

\begin{cor}
  If $\vec M$ is a \lappsl, then $M_0\subset M_\af$ for all $\af$.
\end{cor}

The following is the fundamental lemma of \lappsl s.

\begin{lem}[{\cite[Lem. 2.4]{msfn}}]
  Given $\lm$ regular uncountable, there is a uniformly $\{\lm\}$-definable partition
  of every ordinal $\af$ into a finite number $\da[\lm]$ of intervals
  $(I_{\lm,i}(\af)\suchthat i<\da[\lm])$ such that, for every \lappsl\
  $(M_\af)_{\af<\eta}$, every $\af\leq\eta$, and every $i<\da[\lm]$, the set
  $\{M_\bt\suchthat\bt\in I_{\lm,i}(\af)\}$ is directed \wrt\ inclusion.
  Moreover, if $n\in[1,\om)$ and $\af<\lm^{+n}$, then $\da[\lm]\leq n$.
\end{lem}

\begin{defn}
  Given $\{M_\bt\suchthat\bt\in I_{\lm,i}(\af)\}$ as in the above lemma,
  let $M_{\af,i}$ denote its union.
\end{defn}
By the above lemma, $M_{\af,i}\subh$ and
$\bigcup_{\bt<\af}M_\bt=\bigcup_{i<\da[\lm]}M_{\af,i}$.
Therefore, at any given stage of a transfinite construction
where each step adds an object of size less than $\lm$,
we can collect everything constructed in prior stages
into finitely many structures with strong closure properties.
And for constructions of length $\lm^{+n}$,
``finitely'' improves to ``at most $n$.''

Unrelated to \lappsl s, we will also need the next lemma,
which essentially says that any club subset of $[\ka^{+n}]^{\leq\ka}$
contains $n$ elements in general position \wrt\ inclusion.

\begin{lem}\label{incomparable}
  Given $1\leq n<\om$ and infinite cardinals $\ka$ and $\mu$,
  \tfae.
  \begin{enumerate}
  \item\label{incompsize} $\mu\geq\ka^{+n}$.
  \item\label{incompcof}
    For each directed $\mc{D}\subset[\mu]^{\leq\ka}$ with union $\mu$,
    there exist $A_0,\ldots,A_{n-1}\in\mc{D}$ such that
    $\bigcap_{j\not=i}A_j\not\subset A_i$ for all $i<n$.
  \item\label{incompclub}
    For each club $\mc{E}\subset[\mu]^{\leq\ka}$,
    there exist $B_0,\ldots,B_{n-1}\in\mc{E}$ such that
    $\bigcap_{j\not=i}B_j\not\subset B_i$ for all $i<n$.
  \item\label{incompctbl}
    For each $x\in\bigh$, there exist 
    $M_0,\ldots,M_{n-1}\subh$ such that
    $\card{M_i}=\ka\subset M_i$, $x\in M_i$, and
    $\mu\cap\bigcap_{j\not=i}M_j\not\subset M_i$ for all $i<n$.
  \item\label{incompsub}
    For each $x\in\bigh$, there exist
    $M_0,\ldots,M_{n-1}\subh$ such that
    $\ka\subset M_i$, $x\in M_i$, and
    $\mu\cap\bigcap_{j\not=i}M_j\not\subset M_i$ for all $i<n$.
  \end{enumerate}
\end{lem}

\begin{lem}\label{chainplus}
  If $\rho$ is a cardinal, $M,N\subh$, and
  $\rho\cap M\subset N$, then
  $\rho^+\cap M\cap N$ is downward closed in $\rho^+\cap M$.
\end{lem}
\begin{proof}
  Given $\af\in\rho^+\cap M\cap N$,
  let $f$ be the $\bigwo$-least
  surjection from $\card{\af}$ to $\af$.
  Given $\beta\in\af\cap M$, we have $\beta=f(\gamma)$
  for some $\gamma\in\rho\cap M$.
  But then $\gamma\in N$, and, therefore, $\beta\in N$.
\end{proof}

\begin{proof}[Proof of Lemma~\ref{incomparable}]
  \eqref{incompsize}$\imp$\eqref{incompsub}: 
  Given $0<i\leq n$ and $M_{n-1},\ldots,M_i\in\bigh$,
  choose $M_{i-1}\subh$ of size $\ka^{+i-1}$
  such that $x,M_{n-1},\ldots,M_i\in M_{i-1}$ 
  and $\ka^{+i-1}\subset M_{i-1}$.
  Next, given $M_{n-1},\ldots,M_0$ so constructed,
  suppose $s\subset n$. We will prove that
  $\card{D_m}\geq\ka^{+m}$ for all $m\leq n$ where
  \[D_m=\ka^{+n}\cap\bigcap_{m\leq i\not\in s}M_i
  \setminus\bigcup_{m\leq i\in s}M_i.\]
  If $m=n$, then $D_m=\ka^{+m}$. So,
  suppose $m<n$ and $\card{D_{m+1}}\geq\ka^{+m+1}$.
  If $m\in s$, then $D_m=D_{m+1}\setminus M_m$, which has
  size $\card{D_{m+1}}$ because $\card{M_m}<\card{D_{m+1}}$.
  If $m\not\in s$, then $D_m=D_{m+1}\cap M_m$, which
  has size $\ka^{+m}$ because $\ka^{+m}\subset M_m$
  and $M_m$ knows that $\card{D_{m+1}}\geq\ka^{+m}$.
  By (backwards) induction, $\card{D_0}\geq\ka$.
  Hence, for each $i<n$, letting $s=\{i\}$,
  we obtain $\mu\cap\bigcap_{j\not=i}M_j\not\subset M_i$.
  
  \eqref{incompsub}$\imp$\eqref{incompsize}:
  Given $\ka\subset M_i\subh$ for all $i<n$,
  we will inductively construct a permutation $\tau$ of $n$ such that
  \(\displaystyle\ka^{+n-1}\cap\bigcap_{i<n-1}M_{\tau(i)}
  \subset M_{\tau(n-1)}\).
  First, observe that $\{\ka^+\cap M_i\suchthat i<n\}$ is a chain.
  Given $m<n$, inductively assume that we have
  $\sigma_m\colon m\rightarrow n$ injective and
  \[\left\{\ka^{+m+1}\cap M_i\cap\bigcap_{j<m}M_{\sigma_m(j)}
  \Suchthat i\in n\setminus\ran(\sigma_m)\right\}\]
  a chain. Choose $i\in n\setminus\ran(\sigma_m)$
  so as to obtain the minimum of the above chain;
  extend $\sigma_m$ to $\sigma_{m+1}$ by declaring
  $\sigma_{m+1}(m)=i$.
  By Lemma~\ref{chainplus}, 
  \[\left\{\ka^{+m+2}\cap M_i\cap M_h\cap\bigcap_{j<m}M_{\sigma_{m}(j)}
  \Suchthat h\in n\setminus\ran(\sigma_{m+1})\right\}\]
  consists of downward closed subsets of
  $\ka^{+m+2}\cap M_i\cap\bigcap_{j<m}M_{\sigma_{m}(j)}$,
  and so is a chain.
  Having thus preserved our inductive hypothesis,
  we now declare $\tau=\sigma_n$. By construction,
  \[\ka^{+n-1}\cap M_{\tau(n-2)}\cap\bigcap_{j<n-2}M_{\tau(j)}\subset
  \ka^{+n-1}\cap M_{\tau(n-1)}\cap\bigcap_{j<n-2}M_{\tau(j)}.\]

  \eqref{incompsub}$\imp$\eqref{incompctbl}:
  Given $\vec M$ as in \eqref{incompsub},
  choose $N\subh$ such that $x,\vec M\in N$
  and $\card{N}=\ka\subset N$. Let $P_i=M_i\cap N$. Since
  $N$ knows that $E_i=\mu\cap\bigcap_{j\not=i}M_j\setminus M_i$
  is nonempty, $E_i$ intersects $N$. Thus,
  $\mu\cap\bigcap_{j\not=i}P_j\not\subset P_i$.

  \eqref{incompctbl}$\imp$\eqref{incompsub}: Trivial.

  \eqref{incompctbl}$\imp$\eqref{incompclub}:
  Given a club $\mc{E}\subset[\mu]^{\leq\ka}$,
  if $\mc{E}\in M\subh$ and $\card{M}=\ka\subset M$,
  then $\mu\cap M\in\mc{E}$.
  
  \eqref{incompclub}$\imp$\eqref{incompctbl}:
  The following set is club in $[\mu]^{\leq\ka}$.
  \[\{\mu\cap M\suchthat x\in M\subh
  \text{ and }\card{M}=\ka\subset M\}\]

  \eqref{incompcof}$\imp$\eqref{incompclub}: Trivial.

  \eqref{incompclub}$\imp$\eqref{incompcof}:
  Given a directed $\mc{D}\subset[\mu]^{\leq\ka}$
  with union $\mu$,
  let $\mc{E}$ be the set of all unions
  of $\ka$-sized subsets of $\mc{D}$,
  which is club in $[\mu]^{\leq\ka}$.
  Given $\vec B$ as in \eqref{incompclub}, choose
  $x_i\in\bigcap_{j\not=i}B_j\setminus B_i$ for all $i<n$.
  Then choose $\vec A\in\mc{D}^n$ such that
  $\{x_j\suchthat j\not=i\}\in A_i\subset B_i$.
\end{proof}

\subsection{Half of the main theorem} 
\begin{lem}\label{kahaydon}
  Given a regular infinite cardinal $\ka$ and a Boolean algebra $A$, \tfae.
  \begin{enumerate}
  \item $A$ is tightly $\ka$-filtered.
  \item For some transfinite sequence $(B_\af)_{\af<\eta}$ of subalgebras of $A$ each
    of size at most $\ka$, we have $A=\gen{\bigcup\vec B}$ and, for all $\af<\eta$,
    \[\gen{\bigcup_{\bt<\af}B_\bt}\leq_\ka\gen{\bigcup_{\bt<\af+1}B_\bt}.\]
  \end{enumerate}
\end{lem}
\begin{proof}
  By definition, (1)$\Rightarrow$(2). The converse for $\ka=\alo$ is part of
  \cite[Thm. 1.3.2]{hs}, though the original (topologically stated) result that
  (2) implies projectivity is due to Haydon \cite{haydon}.
  The converse for $\ka=\al_1$ is implied by Koppelberg's \cite[Thm. 2.5]{koppsigma}.
  As observed by Geschke in \cite{geschke}, Koppelberg's proof applies to all
  regular infinite $\ka$.
\end{proof}

\begin{thm}\label{tightgameclub}
  Given a regular infinite cardinal $\ka$ and a Boolean algebra $A$, \tfae.
  \begin{enumerate}
  \item $A$ is tightly $\ka$-filtered.
  \item There is a club $\mc{E}\subset[A]^{\leq\ka}$ such that
    $\gen{\bigcup\mc{B}}\leq_\ka A$ for all $\mc{B}\subset\mc{A}$.
  \item For every nonzero cardinal $\tau$, 
    II has an uncoordinated winning strategy for $\mc{G}(A,\ka,\tau)$.
  \item For every $n\in[1,\om)$ satisfying $\ka^{+n}\leq\card{A}$, 
    II has an uncoordinated winning strategy for $\mc{G}(A,\ka,n)$.
  \item There is a club $\mc{E}\subset[A]^{\leq\ka}$ such that $\gen{\bigcup\mc{B}}\leq_\ka A$
    for every finite $\mc{B}\subset\mc{E}$ satisfying $\ka^{+\card{\mc{B}}}\leq\card{A}$.
  \end{enumerate}
\end{thm}
\begin{proof}
  (1)$\Leftrightarrow$(2) follows from a stronger result\footnote{
    ``Club'' is improved to ``additive $\ka$-skeleton.''}
  proved by Geschke for $\ka>\alo$ \cite[Thm. 2.5(v)]{geschke} .
  For case $\ka=\alo$, see \cite[Thm. 1.3.2(4)]{hs},
  which was proved earlier in topological terms by \scepin\ \cite[Thm. 26]{scfunct}.

  (2)$\Rightarrow$(3) follows from Lemma~\ref{geschketogame}.
  (3)$\Rightarrow$(4) is trivial.
  (4)$\Leftrightarrow$(5) follows from Lemma~\ref{geschketogame} and the fact that
  a countable intersection of clubs is a club.  It remains to prove (5)$\Rightarrow$(1).
  
  Let $\mc{E}\subset[A]^{\leq\ka}$ be as in (5) and let $(M_\af)_{\af<\card{A}}$
  be a long $\ka^+$-approximation sequence with $A,\ka,\mc{E}\in M_0$.
  If $\ka>\alo$, then choose each $M_\af$ to be the union of
  a long $\ka$-approximation sequence $(N^{(\af)}_\bt)_{\bt<\ka}$.
  Given $\af<\card{A}$, let $n=\da[\ka^+]$,
  which satisfies $\ka^{+n}\leq\card{A}$.
  By Lemma~\ref{kahaydon}, \istst\ $B\leq_\ka\gen{B\cup C}$ where
  $B=\gen{A\cap\bigcup_{i<n}M_{\af,i}}$ and $C=A\cap M_\af$.
  Therefore, it is enough to prove the stronger claim $B\leq_\ka A$.
  
  Suppose that $x\in A\cap M_\af$ and $I=B\cap{\downarrow}x$.
  By elementarity, \istst\ $M_\af\models\cf(I)<\ka$.
  Because $A\cap M_{\af,i}\cap M_\af\in\mc{E}$ for each $i<n$,
  we have $\cf(I\cap M_\af)<\ka$.
  So, if $\ka=\alo$, then $M_\af\models\exists\max(I)$.
  On the other hand, if $\ka\geq\al_1$, then, for some $\bt<\ka$,
  $M_\af$ satisfies ``$N^{(\af)}_\bt\cap I$ is cofinal in $I$.''
  Thus, $M_\af\models\cf(I)<\ka$ in both cases.
\end{proof}

\subsection{Half of the main example}
\begin{defn}\label{hypergraphdef}
  Given $0<d<n<\lm$, let $X_{\lm,d,n}\subset\powset\left([\lm]^d\right)$
  consist of all $\Gamma\subset[\lm]^d$ such that $[\sigma]^d\not\subset\Gamma$
  for all $\sigma\in[\lm]^n$. For each $\tau\in[\lm]^d$, let
  \[f(\tau)=f_{\lm,d,n}(\tau)=\{\Gamma\in X_{\lm,d,n}\suchthat\tau\in\Gamma\},\]
  and let $C_{\lm,d,n}$ be the Boolean closure of
  $f[[\lm]^d]$ in $\powset(X_{\lm,d,n})$. Topologize $X_{\lm,d,n}$
  by declaring $C_{\lm,d,n}$ to be a base of open sets.
  If $\ka$ is a regular infinite cardinal, then let $C_{\ka,d}=C_{\ka^{+d},d,d+1}$.
\end{defn}

\begin{lem}\label{trivialcliquecover}
  Given a cardinal $\mu$ and sets $a$ and $b_i$ for $i<\mu$ such that
  $\card{a}\geq\mu$ and $[a]^\mu\subset\bigcup_i[b_i]^\mu$,
  there exists $k<\mu$ such that $a\subset b_k$.
\end{lem}
\begin{proof}
  We prove the contrapositive. Suppose $\gm_i\in a\setminus b_i$ for each $i$.
  Extend $\{\gm_i\suchthat i<\mu\}$ to some $c\in[a]^\mu$.
  Then $c$ witnesses that $[a]^\mu\not\subset\bigcup_i[b_i]^\mu$.
\end{proof}

\begin{thm}\label{hypergraphbad}
  Given $1\leq d<n<\om\leq\ka<\ka^{+d}\leq\lm$ and
  a club $\mc{E}\subset[C_{\lm,d,n}]^{\leq\ka}$, there is $\mc{B}\in[\mc{E}]^d$
  such that $\gen{\bigcup\mc{B}}\not\leq_\ka C_{\lm,d,n}$.
\end{thm}
\begin{proof}
  By Lemma~\ref{incomparable}, there exist $M_0,\ldots,M_{d-1}\subh$
  such that $\ka\cup\{\mc{E},f\}\subset M_i$, $\card{M_i}=\ka$, and
  $\lm\cap\bigcap_{h\not=i}M_h\not\subset M_i$ (where $\bigcap\vn=\bigh$).
  \Istst\ $D\not\leq_\ka C_{\lm,d,n}$ where $D=\gen{C_{\lm,d,n}\cap\bigcup\vec M}$.
  For each $i<d$, choose finite nonempty sets
  \[\rho_i\subset\lm\cap\left(\bigcap_{h\not=i}M_h\right)\setminus M_i\]
  such that $\card{\bigcup\vec\rho}=n-1$ (which is possible because $d<n$).
  Then define $\zeta=\bigcup\vec\rho$ and
  \begin{equation}\label{uwitnessbad}
    u=\bigwedge\left\{f(\sigma)\suchthat\sigma\in[\zeta]^d\right\}.
  \end{equation}
  Given $\gm<\ka$ and $u\leq v_\bt\in D$ for $\bt<\gm$,
  \istf\ $w\in D$ such that $u\perp w\not\perp v_\bt$ for all $\bt$.
  Towards this end, choose $\vec T\in\left([\lm]^{<\ka}\right)^d$ such that
  $T_i\subset M_i$ and $v_\bt\in E=\gen{\bigcup_if[[T_i]^d]}$ for all $\bt$.
  Then choose $\dl\in\ka\setminus\bigcup\vec T$,
  let $\eta=\zeta\cup\{\dl\}$, and let
  \begin{equation}\label{wwitnessbad}
    w=\bigwedge\left\{f(\tau)\suchthat\dl\in\tau\in[\eta]^d\right\}.
  \end{equation}
  For each $\tau$ that is as in \eqref{wwitnessbad}, we have
  $\card{\tau\setminus\{\dl\}}<d$ and, hence, some $i<d$ such that
  \[\tau\subset\{\dl\}\cup\bigcup_{h\neq i}\rho_h\subset M_i.\]
  Thus, $w\in D$.
  Moreover, $u\perp w$ because $\card{\eta}=n$ and each $\nu\in[\eta]^d$ is either
  some $\tau$ as in \eqref{wwitnessbad} or some $\sigma$ as in \eqref{uwitnessbad}.
  Fixing $\bt<\gm$, it remains only to show that $w\not\perp v_\bt$.

  Since $v_\bt\geq u>0$, we have $v_\bt\geq\bigvee_{l<L}v_{\bt,l}\geq u>0$
  where each $v_{\bt,l}$ satisfies $v_{\bt,l}\not\perp u$
  and is of the form $\bigwedge_\xi(-1)^{e_l(\xi)}f(\xi)$
  where $e_l$ is a finite partial function from $\bigcup_i[T_i]^d$ to 2.
  \Istst\ $w\not\perp v_{\bt,0}$. Seeking a contradiction,
  suppose that $w\perp v_{\bt,0}$. Let $\eps=e_0$.
  Because $\dl\not\in\bigcup\vec T$ and $\dom(\eps)\subset\bigcup_i[T_i]^d$,
  no $\tau\in\inv{\eps}[\{1\}]$ can be as in \eqref{wwitnessbad}.
  Therefore, some $\xi\in[\lm]^n$ satisfies
  \begin{equation}\label{wvperp}
    [\xi]^d\subset\inv{\eps}[\{0\}]
    \cup\left\{\tau\suchthat\dl\in\tau\in[\eta]^d\right\}
    \subset\bigcup_{i<d}[M_i]^d.
  \end{equation}
  By Lemma~\ref{trivialcliquecover}, we have $\xi\subset M_i$ for some $i<d$.
  Therefore, $\rho_i\cap\xi=\vn$.
  Since $\card{\xi}=n>\card{\eta\setminus\rho_i}$, we conclude that
  $\xi\not\subset\eta$. Choosing $\af\in\xi\setminus\eta$, we have
  \[\{\nu\in[\xi]^d\suchthat\af\in\nu\}\subset\inv{\eps}[\{0\}]\]
  because of \eqref{wvperp}. Therefore, $d=1$ implies
  \[[\{\af\}\cup\zeta]^d=\{\af\}\cup[\zeta]^d\subset\inv{\eps}[\{0\}]\cup[\zeta]^d,\]
  which contradicts $u\not\perp v_{\bt,0}$. Thus, $d\geq 2$.
  
  We have $[\xi]^d\not\subset\inv{\eps}[\{0\}]$ because $u\not\perp v_{\bt,0}$.
  Combining this fact with \eqref{wvperp}, we have $\dl\in\xi$.
  Since $d\geq 2$, we may extend $\{\af,\dl\}$ to some
  $\tau\in[\xi]^d$, which is not is not in $\inv{\eps}[\{0\}]$
  because $\dl\not\in\bigcup\vec T$. Again applying \eqref{wvperp},
  we have $\af\in\tau\subset\eta$, which contradicts $\af\in\xi\setminus\eta$.
\end{proof}
\begin{cor}
  Given a club $\mc{E}\subset[C_{\ka,d}]^{\leq\ka}$, there is $\mc{B}\in[\mc{E}]^d$
  such that $\gen{\bigcup\mc{B}}\not\leq_\ka C_{\ka,d}$. Hence,
  II does not have an uncoordinated winning strategy for $\mc{G}(C_{\ka,d},\ka,d)$.
\end{cor}

\section{Open-open games for teams}\label{secopenopen}
In this section we define a team version of the open-open game of
Daniels, Kunen, and Zhou (henceforth, DKZ) and use it to give
a new characterization of Cohen algebras. We also observe that projective
Boolean algebras can be characterized in terms of a strengthening of DKZ's
topological property of being very I-favorable.

\subsection{I-favorable spaces for teams}
\begin{defn}
  Define the \emph{open-open} game $\mc{O}_\tau(X)$ for space $X$ and
  nonzero finite team size $\tau$ as follows.
  There are $\om$ rounds. In round $n$,
  player $k$ of team I plays an open set $U_{n,k}$ and
  player $k$ of team II plays an open set $V_{n,k}$.
  Team I wins iff at least one of the following conditions holds.
  \begin{enumerate}
  \item[(O1)] For some $n<\om$,
    $\bigcap_{k<\tau}(U_{n,k}\cap V_{n,k})=\vn\neq\bigcap_{k<\tau}U_{n,k}$.
  \item[(O2)] $\bigcup_{n<\om}\bigcap_{k<\tau}V_{n,k}$ is dense in $X$.
  \end{enumerate}
\end{defn}
For team size 1, the above game is the open-open game of DKZ~\cite{dkz}
except that in the latter I also wins if $V_{n,0}\not\subset U_{n,0}\neq\vn$.
However, this distinction does not change who has a winning strategy
because II shrinking each $V_{n,k}$ to $U_{n,k}\cap V_{n,k}$
has no effect on (O1) and makes (O2) more difficult for I to achieve.
If $\tau>1$ and $X$ is a nontrivial Hausdorff space, then
II does not have an uncoordinated winning strategy.
But it does make sense to ask when team I can guarantee victory
without coordination.

\begin{defn}
  Say $X$ is \emph{I-favorable for team size $\tau$} iff I
  has an uncoordinated winning strategy for $\mc{O}_\tau(X)$.
  Say that $X$ is \emph{I-favorable} iff $X$ is I-favorable for team size 1.
\end{defn}

If $\sigma$ is a $\mc{O}_\tau(X)$-strategy for I
and $1\leq\upsilon<\tau$, then consider the following
$\mc{O}_\upsilon(X)$-strategy $\sigma'$ for I.
To make $\sigma$ applicable, Team I pretends that II
has team size $\tau$ and that II plays
$V_{n,k}=V_{n,\upsilon-1}$ for all all $k\in[\upsilon,\tau)$.
Players $0,\ldots,\upsilon-2$ of team I each play according to $\sigma$;
player $\upsilon-1$ plays $U_{n,\upsilon-1}\cap\cdots\cap U_{n,\tau-1}$.
If $\sigma$ is winning, then so $\sigma'$;
if $\sigma$ is uncoordinated, then so is $\sigma'$.

In the opposite direction, if $\sigma$ is a winning $\mc{O}_1(X)$-strategy
for I (or for II), then the winning team can also win $\mc{O}_\tau(X)$
for any $\tau$. To apply $\sigma$, each player of the winning team
pretends that the opposing team is a single player playing $\bigcap_{k<\tau}V_{n,k}$
($\bigcap_{k<\tau}U_{n,k}$ if $\sigma$ is for II);
each player of the winning team plays $\bigcap_{k<\tau}U_{n,k}$
($\bigcap_{k<\tau}V_{n,k}$ if $\sigma$ is for II).
However, this modified $\sigma$ is coordinated.
Indeed, we will give examples where $X$ is I-favorable for team size $\tau$
but not for team size $\tau+1$.

The topology of a space $X$ is a complete Boolean algebra \wrt\
inclusion. DKZ naturally generalized the open-open game
to arbitrary Boolean algebras and arbitrary posets.
Here are the analogs for teams.
\begin{defn}
  Define the open-open game $\mc{O}_\tau(A)$ for
  a Boolean algebra $A$ and team size $\tau\in[1,\om)$ as follows.
  There are $\om$ rounds. In round $n$,
  player $k$ of team I plays $a_{n,k}\in A$ and
  player $k$ of team II plays $b_{n,k}\in A$.
  Team I wins iff at least one of the following conditions holds.
  \begin{enumerate}
  \item[(A1)] For some $n<\om$,
    $\bigwedge_{k<\tau}(a_{n,k}\wedge b_{n,k})=0\neq\bigwedge_{k<\tau}a_{n,k}$.
  \item[(A2)] $\{\bigwedge_{k<\tau}b_{n,k}\suchthat n<\om\}\setminus\{0\}$
    is predense in $A$.
  \end{enumerate}
\end{defn}

\begin{defn}
  Define the open-open game $\mc{O}_\tau(\mb{P})$ for
  poset $\mb{P}$ and team size $\tau\in[1,\om)$ as follows.
  There are $\om$ rounds. In round $n$,
  player $k$ of team I plays $a_{n,k}\in\mb{P}$ and
  player $k$ of team II plays $b_{n,k}\in\mb{P}$.
  Team I wins iff at least one of the following conditions holds.
  \begin{enumerate}
  \item[(P1)] For some $n<\om$, there is a common extension of
    $a_{n,0},\ldots,a_{n,\tau-1}$ but no common extension of
    $a_{n,0},b_{n,0},\ldots,a_{n,\tau-1},b_{n,\tau-1}$.
  \item[(P2)] $\{c\in\mb{P}\suchthat\exists n\ \forall k\ c\leq b_{n,k}\}$
    is dense in $\mb{P}$.
  \end{enumerate}
\end{defn}

DKZ characterized the I-favorability of posets as the existence of
a club of countable complete suborders.
Balcar, Jech, and Zapletal (henceforth, BJZ) proved that I-favorability
for Boolean algebras is also characterized by II having a winning strategy
for a game similar to the FKS game for $\ka=\om$.
\cite[Thm 4.3]{bjz}
We will extend both of these characterizations to open-open team games.

\begin{defn}\
  \begin{itemize}
  \item Say that a subalgebra $A$ of a Boolean algebra $B$
    is a \emph{regular subalgebra} and write $A\subreg B$ iff
    every nonzero $b\in B$ has a \emph{reduction} to $A$, that is, some
    nonzero $a\in A$ such that for all $c\in A$ we have
    $c\not\perp a\Rightarrow c\not\perp b$.
  \item Say that a subset $\mb{P}$ of poset $\mb{Q}$
    is a \emph{complete suborder} and write $\mb{P}\subset_c\mb{Q}$ iff
    \begin{itemize}
    \item $p_0\perp p_1$ in $\mb{P}$ implies $p_0\perp p_1$ in $\mb{Q}$, and
    \item for each $q\in\mb{Q}$ there exists
    a \emph{reduction} of $q$ to $\mb{P}$, that is, some
    $p\in\mb{P}$ such that for all $r\in\mb{P}$ we have
    $r\not\perp p\Rightarrow r\not\perp q$.
    \end{itemize}
  \item Given a Boolean algebra $A$, let $\overline{A}$ denote its completion.
  \item Given a poset $\mb{P}$,
    \begin{itemize}
    \item let $\mb{P}/{\sim_{\mb{P}}}$ denote its separative quotient,
    \item let $\overline{\mb{P}}$ denote the Boolean completion of
      $\mb{P}/{\sim_{\mb{P}}}$ (that is, the complete Boolean algebra
      consisting of the regular open subsets of $\mb{P}$ ordered by inclusion,
      with each element of $\mb{P}/{\sim_{\mb{P}}}$ naturally identified
      with an element of $\overline{\mb{P}}$), and
    \item let $\gen{\mb{P}}$ denote the Boolean subalgebra of
      $\overline{\mb{P}}$ generated by $\mb{P}/{\sim_{\mb{P}}}$.
    \end{itemize}
  \end{itemize}
\end{defn}
We will use the following facts.
\begin{itemize}
\item $\mb{P}\subset_c\mb{Q}$ iff every set $D$ that is predense in $\mb{P}$
  is also predense in $\mb{Q}$.
\item If $\mb{P}\subset\mb{Q}$ and
  $\mb{P}/{\sim_{\mb{Q}}}\subset_c\overline{\mb{Q}}$,
  then $\mb{P}\subset_c\mb{Q}$.
\item If $A_0$ and $A_1$ are dense subalgebras of $B$ and $B$ is a subalgebra
  of $C$, then $A_0\subreg C$ implies $A_1\subreg B$.
\end{itemize}

\begin{defn}
  Given a poset $\mb{P}$ or a Boolean algebra $A$, respectively define the
  \emph{BJZ game} $\mc{C}_\tau(\mb{P})$ or $\mc{C}_\tau(A)$
  for nonzero finite team size $\tau$ as follows. There are $\om$ rounds.
  In each round, each player of each team plays an element of $\mb{P}$ or $A$.
  For $A$, team II wins iff the Boolean closure of the set of all
  elements played is a regular subalgebra of $A$.
  For $\mb{P}$, team II wins iff the Boolean algebra generated by the set of all
  $\sim_{\mb{P}}$-equivalence classes of elements played is
  a regular subalgebra of $\overline{\mb{P}}$.
\end{defn}

\begin{lem}\label{posetgames}
  Given $1\leq\tau<\om$ and a poset $\mb{P}$, \tfae.
  \begin{enumerate}
  \item I has an uncoordinated winning strategy for $\mc{O}_\tau(\mb{P})$.
  \item II has an uncoordinated winning strategy for $\mc{C}_\tau(\mb{P})$.
  \item There is a club $\mc{D}\subset[\mb{P}]^{\leq\alo}$ such that
    $\bigwedge\vec{\mb{Q}}\subset_c\overline{\mb{P}}$ for all
    $\vec{\mb{Q}}\in\mc{D}^\tau$.
  \end{enumerate}
\end{lem}
Above, $\bigwedge\vec{\mb{Q}}$ denotes
\(\{\bigwedge_i(q_i/{\sim_{\mb{P}}})\suchthat\vec q\in\prod\vec{\mb{Q}}\}\).
\begin{proof}
  (2)$\Rightarrow$(3): Let $\rho$ be an uncoordinated winning strategy
  for II in $\mc{C}_\tau(\mb{P})$. Suppose $\rho\in M_k\subh$ and
  $M_k$ is countable for each $k<\tau$.
  \Istst\ $\gen{\mb{P}\cap\bigcup\vec{M}}\subreg\overline{\mb{P}}$.
  Let team I play against $\rho$ as follows. For each $k<\tau$,
  player $k$ enumerates $\mb{P}\cap M_k$. Since $\rho$ is uncoordinated
  and $M_k$ is $\rho$-closed, every play by player $k$ of team II
  will be in $\mb{P}\cap M_k$. Since $\rho$ is winning,
  $\gen{\mb{P}\cap\bigcup\vec{M}}\subreg\overline{\mb{P}}$.

  (3)$\Rightarrow$(1): Fix $\mc{D}$ as in (3) and fix a surjection
  $f\colon\om\to\om^{\tau+1}$ such that $f(n)(\tau)\leq n$ for all $n$.
  Define an uncoordinated strategy for I in $\mc{O}_\tau(\mb{P})$ as follows.
  In round $n$, each player $k$ of team I chooses $g_{n,k}\colon\om\to\mb{P}$
  such that $$\bigcup_{m<n}\ran(g_{m,k})\subset\ran(g_{n,k})\in\mc{D}$$
  and such that any previous plays $b_{m,k}$ for $m<n$ by opposing player $k$
  of team II are also in $\ran(g_{n,k})$. Then player $k$ of team I plays
  $a_{n,k}=g_{f(n)(\tau),k}(f(n)(k))$. Let $\mb{Q}_k=\bigcup_{n<\om}\ran(g_{n,k})$.
  After $\om$ rounds, every $\vec{q}\in\prod_{k<\tau}\mb{Q}_k$ has been played
  by team I in some round.  Therefore, if II has not lost according to (P1),
  then the meets $\bigwedge_{k<\tau}(b_{n,k}/{\sim})$ of II's plays
  form a predense subset $E$ of $\bigwedge\vec{\mb{Q}}$.
  By (3), $\bigwedge\vec{\mb{Q}}\subset_c\overline{\mb{P}}$.
  Therefore, every predense subset $D$ of $\bigwedge\vec{\mb{Q}}$ generates
  a dense ideal $$\{c\in\mb{P}\suchthat\exists d\in D\ c\leq D\}$$ of $\mb{P}$.
  Thus, I has an uncoordinated winning strategy.

  (1)$\Rightarrow$(3): Let $\rho$ be an uncoordinated winning strategy
  for I in $\mc{O}_\tau(\mb{P})$. Let $\mb{Q}=\bigwedge_k(M_k\cap\mb{P})$
  where $\rho\in M_k\subh$ and $M_k$ is countable for each $k<\tau$.
  \Istst\ $\mb{Q}\subset_c\overline{\mb{P}}$. Therefore,
  supposing that $D$ is predense in $\mb{Q}$, \istst\ it is also predense
  in $\overline{\mb{P}}$. Seeking a contradiction,
  suppose that some $c\in\overline{\mb{P}}$ is incompatible with every $d\in D$.
  Then let team II play against $\rho$ as follows.
  In round $n$, inductively assume that team II has played
  $\vec{b}_m\in\prod\vec M$ for each $m<n$. Then team I will play some
  $\vec{a}_n\in\prod\vec M$ in round $n$ because $\rho$ is uncoordinated
  and each $M_k$ is $\rho$-closed. Let team II respond
  with $\vec{b}_n=\vec{a}_n$ if $\bigwedge\vec{a}_n=0$
  or, if $\bigwedge\vec{a}_n\neq0$, with $\vec{b}_n\in\prod\vec M$ such that
  $\bigwedge\vec{a}_n\geq\bigwedge\vec{b}_n\leq d$ for some $d\in D$.
  In both cases, II avoids losing according to (P1) and arranges that
  $\bigwedge\vec{b}_n\perp c$. Thus, II may play against $\rho$
  such that the ideal of $\mb{P}$ generated by $\bigwedge\vec{b}_n$ for $n<\om$
  will not be predense, in contradiction with (1).
  
  (3)$\Rightarrow$(2): In the game $\mc{C}_\tau(\mb{P})$,
  let $\mb{Q}_k$ denote the set $\bigcup_{n<\om}\{a_{n,k},b_{n,k}\}$
  of all plays made by player $k$ of team I or player $k$ of team II.
  Using a bookkeeping argument like in the proof of (3)$\Rightarrow$(1)
  (but simpler), II can, without coordination, ensure that
  $\mb{Q}_k\in\mc{D}$ for each $k<\tau$.
  By (3), $\gen{\bigwedge\vec{\mb{Q}}}\subreg\overline{\mb{P}}$.
  Therefore, II has an uncoordinated winning strategy for $\mc{C}_\tau(\mb{P})$.
\end{proof}
For team size 1, we have $\bigwedge\vec{\mb{Q}}\subset_c\overline{\mb{P}}$
iff $\mb{Q}_0\subset_c\mb{P}$. Therefore, Lemma~\ref{posetgames}
generalizes Theorem~1.6 of \cite{dkz}.

\begin{defn}
  Given a space $X$, let $\op(X)$ denote the poset of
  nonempty open subsets of $X$ ordered by inclusion.
\end{defn}

\begin{cor}\label{topteamopenopen}
  Given $1\leq\tau<\om$ and a nonempty topological space $X$, \tfae.
  \begin{itemize}
  \item I has an uncoordinated winning strategy for $\mc{O}_\tau(X)$.
  \item II has an uncoordinated winning strategy for $\mc{C}_\tau(\op(X))$.
  \item There is a club $\mc{D}\subset[\op(X)]^{\leq\alo}$ such that
    $\bigwedge\vec{\mb{Q}}\subset_c\overline{\op(X)}$ for all
    $\vec{\mb{Q}}\in\mc{D}^\tau$.
  \end{itemize}
\end{cor}
\begin{proof}
  Given Lemma~\ref{posetgames} applied to $\op(X)$, \istst\ removing
  the option to play $U_{n,k}=\vn$ or $V_{n,k}=\vn$ from the game $\mc{O}_\tau(X)$
  does not change whether $I$ has an uncoordinated winning strategy.
  First, it never helps team I to play $U_{n,k}=\vn$
  because player $k$ of team II can respond with $V_{n,k}=\vn$, which prevents
  team I from making progress towards winning condition (O2).
  Thus, allowing $U_{n,k}=\vn$ does not help team I, coordinated or not,
  to defeat team II, coordinated or not.
  Second, if every $U_{n,k}$ is nonempty, then allowing $V_{n,k}=\vn$
  does not help a coordinated team II win:
  if $\bigcap_kU_{n,k}\neq\vn$, then $V_{n,k}=\vn$ for any $k$
  is an immediate loss for team II; if $\bigcap_kU_{n,k}=\vn$,
  then team II can still achieve the optimal outcome
  $\bigcap_kV_{n,k}=\vn$ for round $n$ by playing $V_{n,k}=U_{n,k}\neq\vn$.
\end{proof}

\begin{cor}\label{boolteamopenopen}
  Given $1\leq\tau<\om$ and a Boolean algebra $A$, \tfae.
  \begin{enumerate}
  \item I has an uncoordinated winning strategy for $\mc{O}_\tau(A)$.
  \item II has an uncoordinated winning strategy for $\mc{C}_\tau(A)$.
  \item There is a club $\mc{D}\subset[A]^{\leq\alo}$ such that
    $\gen{\bigcup\vec B}\subreg A$ for all $\vec B\in\mc{D}^\tau$.
  \end{enumerate}
\end{cor}
\begin{proof}
  Apply Lemma~\ref{posetgames} to the separative poset $A\setminus\{0\}$,
  noting that every club subset of $[A\setminus\{0\}]^{\leq\alo}$ contains
  a club whose elements are all of the form $B\setminus\{0\}$ where
  $B$ is a subalgebra of $A$.
\end{proof}
In \cite{hs}, Heindorf and Shapiro call Boolean algebras satisfying (3) of
Corollary~\ref{boolteamopenopen} for team size 1 \emph{regularly filtered}.
In \cite{bjz}, $\pi$-homogeneous Boolean algebras satisfying
(3) of Corollary~\ref{boolteamopenopen} for team size 1
are called \emph{semi-Cohen}. Theorem~4.3(a,d,e) of \cite{bjz} is
Corollary~\ref{boolteamopenopen} restricted to team size 1
and $\pi$-homogeneous Boolean algebras.

\subsection{Characterizing Cohen algebras}
\begin{defn}\ 
  \begin{itemize}
  \item Say that Boolean algebras $A,B$ are \emph{cocomplete} if
    $\overline{A}\cong\overline{B}$.
  \item Say that a Boolean algebra $A$ is a \emph{Cohen algebra}
    if it cocomplete with an infinite free Boolean algebra.
  \end{itemize}
\end{defn}
Thus, a forcing $\mb{P}$ is equivalent to a Cohen forcing
$\mathrm{Fn}(J,2)$ with $J$ infinite iff $\gen{\mb{P}}$ is a Cohen algebra.
Shapiro proved the following characterization of Cohen algebras.

\begin{defn} Given a Boolean algebra $A$:
  \begin{itemize}
  \item Say that $A$ is \emph{weakly projective}
    if it cocomplete with a projective Boolean algebra. \cite{hs}
  \item Say that $A$ is \emph{$\pi$-homogeneous}
    if $A$ is infinite and $\pi(A\restrict a)=\pi(A)$
    for all $a\in A\setminus\{0\}$.
  \end{itemize}
\end{defn}

\begin{thm}[{\cite[Ch. 5]{hs}}]\label{weakprojchain}
  A Boolean algebra is Cohen iff it is weakly projective and
  $\pi$-homogeneous. Moreover, a Boolean algebra $A$ is weakly projective
  iff there is a transfinite sequence $(B_\af)_{\af<\eta}$
  of subalgebras of $A$ such that $A=\gen{\bigcup\vec B}$ and,
  for each $\af<\eta$, $\card{B_\af}\leq\alo$ and
  $$\gen{\bigcup_{\bt<\af}B_\bt}\subreg\gen{\bigcup_{\bt<\af+1}B_\bt}.$$
\end{thm}
In the above theorem, the $\pi$-homogeneity requirement is less
restrictive than it appears because every weakly projective Boolean algebra
is co-complete with a product of countably many Cohen algebras.
\cite[Thm. 5.2.2]{hs}
The topological equivalent of this fact is that a Stone space is
co-absolute with a Dugundji space iff it is co-absolute with a countable sum
of powers of 2.

Jech characterized weakly projective Boolean algebras in terms of
clubs of countable regular subalgebras.
Arguing like in the proof of Theorem~\ref{tightgameclub},
we will weaken the closure requirement he puts on the club.

\begin{thm}[Jech]\label{jechclubadd}
  A Boolean algebra $A$ is weakly projective iff there is a club
  $\mc{D}\subset[A]^{\leq\alo}$ such that $B\subreg A$ for all $B\in\mc{D}$
  and $\gen{B_0\cup B_1}\in\mc{D}$ for all $B_0,B_1\in\mc{D}$.
\end{thm}
\begin{proof}
  See \cite[Thm 3.2]{bjz} for a proof restricted to
  $\pi$-homogeneous Boolean algebras. See \cite[Thm 5.3.9]{hs}
  for a complete proof and attribution of the result of Jech.
\end{proof}

\begin{thm}\label{weakprojteam}
  Given a Boolean algebra $A$, \tfae.
  \begin{enumerate}
  \item $A$ is weakly projective.
  \item There is a club $\mc{D}\subset[A]^{\leq\alo}$ such that
    $\gen{\bigcup\vec B}\subreg A$ for all $\vec B\in\mc{D}^\tau$,
    for all $\tau\in[1,\om)$ satisfying $\al_\tau\leq\pi(A)$.
  \item For each $\tau\in[1,\om)$ satisfying $\al_\tau\leq\pi(A)$,
    I has an uncoordinated winning strategy for $\mc{O}_\tau(A)$.
  \item For each $\tau\in[1,\om)$ satisfying $\al_\tau\leq\pi(A)$,
    II has an uncoordinated winning strategy for $\mc{C}_\tau(A)$.
  \end{enumerate}
\end{thm}
\begin{proof}
  By Corollary~\ref{boolteamopenopen},
  (2)$\Leftrightarrow$(3)$\Leftrightarrow$(4).
  By Theorem~\ref{jechclubadd}, (1)$\Rightarrow$(2).
  We will show that (2)$\Rightarrow$(1).
  Let $(M_\af)_{\af<\pi(A)}$ be a \lappso\ with $A,\mc{D}\in M_0$,
  and let $B=\gen{A\cap\bigcup_{\af<\pi(A)}M_\af}$. Then $B$ is a dense subalgebra
  of $A$ and, hence, \istst\ $B$ is weakly projective.
  Given $\af<\pi(A)$, let $n=\da[\oml]$, which satisfies $\al_n\leq\pi(A)$.
  By Theorem~\ref{weakprojchain}, \istst\ $E\subreg\gen{E\cup F}$
  where $E=\gen{A\cap\bigcup_{i<n}M_{\af,i}}$ and $F=A\cap M_\af$.
  Therefore, given $p\in A$, \istf\ a reduction of $p$ to $E$.
  By elementarity, it is enough to assume $p\in A\cap M_\af$ and
  find a reduction of $p$ to $E\cap M_\af$.
  Because $A\cap M_{\af,i}\cap M_\af\in\mc{D}$ for each $i<n$,
  there is a reduction of $p$ to $C$ where
  $C=\gen{A\cap M_\af\cap\bigcup_{i<n}M_{\af,i}}$.
  By elementarity again, we have $C=E\cap M_\af$.
\end{proof}

\begin{cor}\label{cohenteam}
  Given a $\pi$-homogeneous Boolean algebra $A$, \tfae.
  \begin{itemize}
  \item $A$ is Cohen.
  \item There is a club $\mc{D}\subset[A]^{\leq\alo}$ such that
    $\gen{\bigcup\vec B}\subreg A$ for all $\vec B\in\mc{D}^\tau$,
    for all $\tau\in[1,\om)$ satisfying $\al_\tau\leq\pi(A)$.
  \item For each $\tau\in[1,\om)$ satisfying $\al_\tau\leq\pi(A)$,
    I has an uncoordinated winning strategy for $\mc{O}_\tau(A)$.
  \item For each $\tau\in[1,\om)$ satisfying $\al_\tau\leq\pi(A)$,
    II has an uncoordinated winning strategy for $\mc{C}_\tau(A)$.
  \end{itemize}
\end{cor}

The Stone dual of a regular Boolean subalgebra
is a semi-open\footnote{$f\colon X\to Y$ is \emph{semi-open} if $f$ maps
  sets with nonempty interior to sets with nonempty interior.}
continuous surjection between Stone spaces.

\subsection{Very I-favorable spaces for teams}
In order to show that every dyadic space\footnote{A Hausdorff space is
  \emph{dyadic} iff it is a continuous image of a power of 2.} is I-favorable,
DKZ introduced the class of very I-favorable spaces,
a subclass of the I-favorable spaces that includes the powers of 2 and,
when restricted to $T_3$ spaces, is closed \wrt\ images of
perfect maps.\footnote{
  DKZ proved closure \wrt\ continuous images of compact Hausdorff
  spaces. Kucharski, Plewik and Valov observed that the same proof applies more
  generally. \cite{kpvvery}}
  
\begin{defn} Given a space $X$:
  \begin{itemize}
  \item If $\mc{E}\subset\op(X)$, then let $\mc{E}\subset_!\op(X)$ indicate that
    for each $\mc{S}\subset\mc{E}$ there exists $\mc{T}\subset\mc{E}$
    such that $\bigcup\mc{T}$ is
    the exterior\footnote{The interior of the complement.} of $\bigcup\mc{S}$.
  \item Say that $X$ is \emph{very I-favorable} iff there is a club
    $\mb{D}\subset[\op(X)]^{\leq\alo}$ such that $\mc{E}\subset_!\op(X)$
    for all $\mc{E}\in\mb{D}$.
  \end{itemize}
\end{defn}

Kucharski, Plewik and Valov (henceforth, KPV)
characterized very I-favorable Hausdorff spaces in terms of
an inverse limit system involving nearly open continuous surjective
bonding maps between second countable $T_0$ spaces
(that may not be Hausdorff or even $T_1$). \cite{kpvvery}

\begin{defn}\
  \begin{itemize}
  \item A subset $N$ of a space $X$ is \emph{nearly open}
    if $N\subset\interior_X\cl_XN$. 
  \item A map $f\colon X\to Y$ between spaces
    is \emph{nearly open} if it maps open sets to nearly open sets.
  \end{itemize}
\end{defn}

The connection between $\subset_!$ and nearly open maps is the following.

\begin{lem}[{\cite[Prop. 2.1]{kpvvery}}]\label{nearlyopen}
    A continuous map $f\colon X\to Y$ between spaces is nearly open iff
    $\{\inv{f}[V]\suchthat V\in\op(Y)\}\subset_!\op(X)$.
\end{lem}

The following example illustrates why KPV needed to consider
non-Hausdorff spaces.

\begin{example}
  Let $M\subh$ be countable and let $\dl=\oml\cap M$.
  Let $X=2^{\oml}$ and let $f\colon X\to Y$ be the quotient map
  induced by identifying points $p,q\in 2^{\oml}$
  iff their neighborhood filters agree on $M$, that is, iff
  $$\{U\in\op(X)\cap M\suchthat p\in U\}=\{U\in\op(X)\cap M\suchthat q\in U\}.$$
  Let $a,b\in X$ where $a(\af)=b(\af)=0$ for all $\af\neq\dl$
  but $a(\dl)=0\neq1=b(\dl)$.
  Then $f(a)\neq f(b)$ because $X\setminus\{a\}\in M$.
  But if $b\in U\in\op(X)\cap M$, then $a\in U$.
  Thus, every neighborhood of $f(b)$ is also a neighborhood of $f(a)$.
\end{example}

KPV also introduced a strengthening of very I-favorable that avoids the
need to consider non-Hausdorff spaces.

\begin{defn} Given a space $X$:
  \begin{itemize}
  \item Let $\Sigma_X$ denote the set of all nonempty cozero\footnote{
    $U\subset X$ is cozero iff $X\setminus U=\inv{f}[\{0\}]$ for some
    continuous $f\colon X\to\mb{R}$.} subsets of $X$.
  \item Say that $X$ is very I-favorable \emph{\wrt\ cozero sets}
    iff there is a club $\mb{D}\subset[\Sigma_X]^{\leq\alo}$ such that
    $\mc{E}\subset_!\op(X)$ for all $\mc{E}\in\mb{D}$.
  \end{itemize}
\end{defn}

KPV characterized Tychonoff spaces very I-favorable for cozero sets in terms of
an inverse limit system involving nearly open continuous surjective
bonding maps between separable metric spaces.
For compact Hausdorff spaces, the situation is even nicer:
a continuous map between compact Hausdorff spaces is open iff it is nearly open.
This allowed KPV to connect their strengthening of very I-favorable
with \scepin's class of openly generated compacta.

\begin{defn} Given a compact Hausdorff space $X$:
  \begin{itemize}
  \item Let $C(X)$ denote the set of all continuous real-valued functions on $X$.
  \item Given a set $\mc{F}$, let $\varsigma^X_{\mc{F}}\colon X\to X/\mc{F}$
    be the quotient map induced by identifying points $p,q\in X$ iff
    $h(p)=h(q)$ for all $h\in C(X)\cap\mc{F}$.
  \item (\scepin) Say that $X$ is \emph{openly generated} iff there
    is a club $\mb{D}\subset[C(X)]^{\leq\alo}$ such that $\varsigma^X_{\mc{E}}$
    is open for all $\mc{E}\in\mb{D}$.
  \end{itemize}
\end{defn}
Note that if $X$ is compact Hausdorff and $\mc{F}$ is a set,
then $X/\mc{F}$ is also compact Hausdorff.

\begin{thm}[{\cite[Cor 4.3]{kpvvery}}]
  A compact Hausdorff space is very I-favorable \wrt\ cozero sets
  iff it is openly generated.
\end{thm}

\begin{cor}
  Given a Stone space $X$ with clopen algebra $A$, \tfae.
  \begin{itemize}
  \item $X$ is very I-favorable \wrt\ cozero sets.
  \item There is a club $\mc{C}\subset[A]^{\leq\alo}$ such that
    $B\leq_{\alo}A$ for all $B\in\mc{C}$.
  \item II has a winning strategy for $\mc{G}(A,\alo,1)$.
  \item $A$ has the FN.
  \end{itemize}
\end{cor}
\begin{proof}
The Stone dual of a continuous open surjection is an $\alo$-subalgebra,
more commonly called a relatively complete subalgebra.
Therefore, a Stone space is openly generated iff its clopen algebra
has a club of countable relatively complete subalgebras.
\end{proof}

The above corollary motivates the following generalization of the FKS team game
$\mc{G}(A,\alo,\tau)$ from Boolean algebras to compact Hausdorff spaces
and an associated generalization of very I-favorable \wrt\ cozero sets.

\begin{defn} Given a compact Hausdorff space $X$ and $\tau\in[1,\om)$:
  \begin{itemize}
  \item Define the \emph{open quotient game} $\mc{Q}(X,\tau)$ as follows.
    In round $n$, team I plays $(f_{n,k})_{k<\tau}\in C(X)^\tau$
    and then II plays $(g_{n,k})_{k<\tau}\in C(X)^\tau$.
    After $\om$ rounds, II wins iff $\varsigma^X_{\mc{E}}$ is open where
    $\mc{E}=\bigcup_{n<\om}\bigcup_{k<\tau}\{f_{n,k},g_{n,k}\}$.
  \item Say that $X$ is \emph{openly $\tau$-generated} iff there
    is a club $\mb{D}\subset[C(X)]^{\leq\alo}$ such that
    $\varsigma^X_{\bigcup\vec{\mc{E}}}$ is open for all $\vec{\mc{E}}\in\mb{D}^\tau$.
  \item Say that $X$ is \emph{$\tau$-very I-favorable \wrt\ cozero sets}
    iff there is a club $\mb{D}\subset[\Sigma_X]^{\leq\alo}$ such that
    for each $\vec{\mc{E}}\in\mb{D}^\tau$ we have
    \[\left\{\bigcap_iU_i\suchthat\vec U\in\prod\vec{\mc{E}}\right\}
    \setminus\{\vn\}\subset_!\op(X).\]
  \end{itemize}
\end{defn}

\begin{lem}\label{exclaimbase}
  If $X$ is a space, $\bigcup\mc{B}=X$, $\mc{B}\subset\mc{E}\subset\op(X)$, 
  and every element of $\mc{E}$ is the union of a subset of $\mc{B}$, then
  $\mc{E}\subset_!\op(X)$ iff $\mc{B}\subset_!\op(X)$.
\end{lem}
\begin{proof}
  Assuming $\mc{E}\subset_!\op(X)$, 
  for each $\mc{S}\subset\mc{B}$ there exists $\mc{T}\subset\mc{E}$
  such that $\bigcup\mc{T}$ is the exterior of $\bigcup\mc{S}$.
  And for some $\mc{U}\subset\mc{B}$, we have $\bigcup\mc{T}=\bigcup\mc{U}$.
  Thus, $\mc{E}\subset_!\op(X)$ implies $\mc{B}\subset_!\op(X)$.

  Conversely,
  for each $\mc{V}\subset\mc{E}$, there exists $\mc{A}\subset\mc{B}$ such
  that $\bigcup\mc{A}=\bigcup\mc{V}$. And, assuming $\mc{B}\subset_!\op(X)$,
  there exists $\mc{C}\subset\mc{B}$ such that $\bigcup\mc{C}$ is
  the exterior of $\bigcup\mc{A}$. Moreover, $\mc{C}\subset\mc{B}\subset\mc{E}$.
  Thus, $\mc{B}\subset_!\op(X)$ implies $\mc{E}\subset_!\op(X)$.
\end{proof}

\begin{thm}\label{cozero}
  Given $1\leq\tau<\om$ and a compact Hausdorff space $X$, \tfae.
  \begin{enumerate}
  \item $X$ is openly $\tau$-generated.
  \item $X$ is $\tau$-very I-favorable \wrt\ cozero sets.
  \item II has an uncoordinated winning strategy for $\mc{Q}(X,\tau)$.
  \end{enumerate}
\end{thm}
\begin{proof}
  (1)$\Rightarrow$(2): Let $\mb{D}\subset[C(X)]^{\leq\alo}$ be a club such that
  $\varsigma^X_{\bigcup\vec{\mc{F}}}$ is open for all $\vec{\mc{F}}\in\mb{D}^\tau$.
  For each $k<\tau$, suppose that $\mb{D}\in M_k\subh$ and $M_k$
  is countable. \Istst\ $\mc{B}\subset_!\op(X)$ where
  \begin{equation}\label{elembase}
    \mc{B}=\left\{\bigcap\vec U
  \suchthat\vec U\in\prod_k(\Sigma_X\cap M_k)\right\}\setminus\{\vn\}.
  \end{equation}
  Letting $\mc{T}$ be the topology generated by $C(X)\cap\bigcup\vec M$,
  the quotient map induced by $\mc{T}$ is open because
  $C(X)\cap M_k\in\mb{D}$ for each $k$.
  So, by Lemma~\ref{nearlyopen}, $\mc{T}\subset_!\op(X)$.
  By elementarity, $\Sigma_X\cap M_k$ is a base for the topology generated
  by the set of the functions $C(X)\cap M_k$, for each $k$.
  Therefore, $\Sigma_X\cap\bigcup\vec M$ is a base for $\mc{T}$;
  hence, $\mc{B}\subset_!\op(X)$ by Lemma~\ref{exclaimbase}.

  (2)$\Rightarrow$(1): Let $\mb{E}\subset[\Sigma_X]^{\leq\alo}$ be a club
  witnessing that $X$ is $\tau$-very I-favorable \wrt\ cozero sets.
  For each $k<\tau$, suppose that $\mb{E}\in M_k\subh$ and $M_k$
  is countable. \Istst\ $\varsigma_{\bigcup\vec M}$ is open.
  Defining $\mc{B}$ as in~\eqref{elembase}, we have $\mc{B}\subset_!\op(X)$
  because $\Sigma_X\cap M_k\in\mb{E}$ for each $k$. By Lemma~\ref{exclaimbase},
  $\mc{T}\subset_!\op(X)$ where $\mc{T}$ is the topology generated by $\mc{B}$.
  By Lemma~\ref{nearlyopen}, $\mc{T}$ induces a nearly open quotient map
  $\varphi$ on $X$.
  By elementarity, the set of functions $C(X)\cap M_k$ generates the
  same topology as the set of cozero sets $\Sigma_X\cap M_k$, for each $k$.
  Therefore, $C(X)\cap\bigcup\vec M$ generates the topology $\mc{T}$ and,
  hence, $\varsigma^X_{\bigcup\vec M}=\varphi$, which is nearly open.
  Since $X/\bigcup\vec M$ is compact Hausdorff,
  $\varsigma^X_{\bigcup\vec M}$ is open.
  
  (1)$\Rightarrow$(3): Given a club $\mb{E}\subset[\Sigma_X]^{\leq\alo}$,
  team II can use uncoordinated bookkeeping to ensure that each of the
  $\tau$ sets of plays $\mc{D}_k=\bigcup_{n<\om}\{f_{n,k},g_{n,k}\}$ is in $\mb{E}$.

  (3)$\Rightarrow$(1): Suppose II has an uncoordinated winning strategy $\rho$
  for $\mc{Q}(X,\tau)$. For any $\rho$-closed subsets $\mc{E}_k$
  for $k<\tau$, team I can play against $\rho$ so as to ensure that
  \(\mc{E}_k=\bigcup_{n<\om}\{f_{n,k},g_{n,k}\}\) for all $k$.
\end{proof}

\begin{cor}\label{teamverystone}
  Given $1\leq\tau<\om$ and a Stone space $X$ with clopen algebra $A$, \tfae.
  \begin{enumerate}
  \item $X$ is $\tau$-very I-favorable \wrt\ cozero sets.
  \item $X$ is openly $\tau$-generated.
  \item II has an uncoordinated winning strategy for $\mc{Q}(X,\tau)$.
  \item There is a club $\mc{E}\subset[A]^{\leq\alo}$ such that
    $\gen{\bigcup\mc{B}}\leq_{\alo} A$ for all $\mc{B}\in[\mc{E}]^{\leq\tau}$.
  \item II has an uncoordinated winning strategy for $\mc{G}(A,\alo,\tau)$.
  \end{enumerate}
\end{cor}
\begin{proof}
  By Theorem~\ref{cozero}, (1)$\Leftrightarrow$(2)$\Leftrightarrow$(3).
  By Lemma~\ref{geschketogame}, (4)$\Leftrightarrow$(5).
  To prove (2)$\Leftrightarrow$(4), suppose that $X\in M_k\subh$ for $k<\tau$.
  By elementarity, $A\cap M_k$ induces the same quotient map on $X$
  as $C(X)\cap M_k$, for each $k$. Therefore, $\gen{A\cap\bigcup\vec M}$
  induces the same quotient map as $C(X)\cap\bigcup\vec M$.
  So, by Stone duality, $\varsigma^X_{\bigcup\vec M}$ is open
  iff $\gen{A\cap\bigcup\vec M}\leq_{\alo}A$.
\end{proof}

\begin{cor}
  Given a Stone space $X$ with clopen algebra $A$, \tfae.
  \begin{enumerate}
  \item $A$ is projective.
  \item $X$ is $\tau$-very I-favorable \wrt\ cozero sets
    for all finite nonzero $\tau$ satisfying $\al_\tau\leq\card{A}$.
  \item $X$ is openly $\tau$-generated
    for all finite nonzero $\tau$ satisfying $\al_\tau\leq\card{A}$.
  \item II has an uncoordinated winning strategy for $\mc{Q}(X,\tau)$
      for all finite nonzero $\tau$ satisfying $\al_\tau\leq\card{A}$.
  \item II has an uncoordinated winning strategy for $\mc{G}(A,\alo,\tau)$
      for all finite nonzero $\tau$ satisfying $\al_\tau\leq\card{A}$.
  \end{enumerate}
\end{cor}
\begin{proof}
  By Corollary~\ref{teamverystone},
  (2)$\Leftrightarrow$(3)$\Leftrightarrow$(4)$\Leftrightarrow$(5).
  By Theorem~\ref{tightgameclub}, (5)$\Leftrightarrow$(1).
\end{proof}

\begin{thm}
  The property of being an openly $\tau$-generated compact Hausdorff space
  is preserved by products and retracts.
\end{thm}
\begin{proof}
  For products, suppose that $Y=\prod_{i\in I}X_i$, each $X_i$ is an
  openly $\tau$-generated compact Hausdorff space, and, for each $k<\tau$,
  that $Y\in M_k\subh$ and $M_k$ is countable.
  Given a point $\vec x$ in a basic open set $\prod\vec U\subset Y$,
  we will show that $\varsigma_{\bigcup\vec M}^Y(\vec x)$ has a neighborhood
  contained in $\varsigma_{\bigcup\vec M}^Y\left[\prod\vec U\right]$.
  Let $F\subset I$ be the (finite) support of $\vec U$.
  For each $i\in F$, let $L_i=\bigcup\{M_k\suchthat i\in M_k\}$.
  Letting $G=F\cap\bigcup\vec M$, we have $\varsigma_{L_i}^{X_i}[U_i]$
  open for each $i\in G$
  and, hence, $\min_jf_{i,j}(x_i)>0$ and $\bigcap_j\supp f_{i,j}\subset U_i$
  for some finite sequence $(f_{i,j})_{j<n_i}\in(C(X_i)\cap L_i)^{<\om}$.
  Letting $\pi_i\colon Y\to X_i$ be the $i$th coordinate projection,
  we have $f_{i,j}\circ\pi_i\in C(Y)\cap L_i$ for each $(i,j)$.
  Therefore, the neighborhood
  \[\bigcap_{i\in G}\bigcap_{j<n_i}\varsigma^Y_{\bigcup\vec M}
  \left[\inv{(f_{i,j}\circ\pi_i)}[(0,\infty)]\right]\]
  of $\varsigma_{\bigcup\vec M}^Y(\vec x)$ is as desired.

  For retracts, suppose that $r\colon X\to Y$ and $e\colon Y\to X$
  are continuous maps between compact Hausdorff spaces,
  that $r\circ e=\id_Y$, that $X$ is openly $\tau$-generated, and,
  for each $k<\tau$, that $r,e\in M_k\subh$ and $M_k$ is countable.
  Given an open $V\subset Y$ and $y\in V$, we will show that
  $\varsigma^Y_{\bigcup\vec M}(y)$ has a neighborhood contained in
  $\varsigma^Y_{\bigcup\vec M}[V]$. Letting $x=e(y)$ and $U=\inv{r}[V]$,
  we have $x\in U$ and $\varsigma^X_{\bigcup\vec M}[U]$ open; hence, we have
  $\min_if_i(x)>0$ and $\bigcap_i\supp f_i\subset U$ for some
  $\vec f\in(C(X)\cap\bigcup\vec M)^{<\om}$.
  We also have $f_i\circ e\in C(Y)\cap\bigcup\vec M$ for each $i$.
  Therefore, the neighborhood 
  \[\bigcap_i\varsigma^Y_{\bigcup\vec M}
  \left[\inv{(f_{i,j}\circ e)}[(0,\infty)]\right]\]
  of $\varsigma_{\bigcup\vec M}^Y(y)$ is as desired.
\end{proof}

\section{Locally finite characterizations of projective Boolean algebras}
\label{seclocfin}
In this section, we complete the proofs of the main theorem and main example
using various properties of tuples of Boolean subalgebras that are
natural weakenings of independence. We also use these properties
to give new characterizations of projective Boolean algebras
that are locally finite in a strong sense. (See Corollary~\ref{projlocfin}
and Subsection~\ref{purefin}.)

\subsection{Commuting tuples of Boolean subalgebras}
An important tool in \scepin's results is the concept of a bicommutative square
of quotient maps between topological spaces, to which Heindorf and Shapiro
applied Stone duality to obtain their concept of commuting pairs of Boolean
subalgebras, which they used to characterize the $\al_1$-FN in \cite[Ch. 4]{hs}.
Geschke also used this concept to prove his characterizations of tightly
$\ka$-filtered Boolean algebras for $\ka\geq\al_1$. \cite[Thm. 2.5]{geschke}\
(As explained in the proof of Lemma~\ref{geschkeunion}, Geschke's proof also works for $\ka=\alo$.)
In this subsection, we introduce a natural higher-arity generalization
of the concept of commuting pairs of subalgebras.
In the next subsection, we use it to give new characterizations of
$\ka$-filtered Boolean algebras and projective Boolean algebras in particular.

\begin{defn}\label{boolcomm}
  Given subalgebras $(A_i)_{i\in S}$ of a Boolean algebra $B$:
  \begin{itemize}
  \item Say that $\vec{A}$ is \emph{independent} iff, for every sequence of
    ultrafilters $\vec{U}\in\prod_i\ult(A_i)$, there is an ultrafilter
    $V\in\ult(B)$ extending $\bigcup\vec U$.
  \item Say that $\vec{A}$ \emph{commutes} iff, for every sequence of
    ultrafilters $\vec{U}\in\prod_i\ult(A_i)$, either there is an ultrafilter
    $V\in\ult(B)$ extending $\bigcup\vec U$, or $\vec{U}$ is \emph{incompatible},
    that is, $U_i\cap A_j\neq U_j\cap A_i$ for some $i,j\in S$.
\end{itemize}
\end{defn}
Thus, the commuting tuple concept is a form of conditional independence.
In category-theoretic terms, a sequence of subalgebras $\vec{A}$ of $B$ is:
\begin{itemize}
\item independent iff the inclusion maps $\id\colon A_i\to B$ mutually extend to
  an injective homomorphism $h\colon C\to B$ for some coproduct
  $C=\coprod\vec A$;
\item commuting iff the inclusion maps $\id\colon A_i\to B$ mutually extend to an
  injective homomorphism $h\colon C\to B$ for some colimit $C$
  of the diagram of all inclusions $\id\colon A_i\cap A_j\to A_i$.
\end{itemize}

Case $\card{S}=2$ of the above definition is equivalent to \cite[Def. 4.1.1]{hs}.
(Lemma~\ref{commpair} will prove this.)
However, the following examples suggest that the relationship 
between commuting pairs and commuting triples is not a simple one.
\begin{example}
  Let $B$ be freely generated by elements $g_0,g_1,g_2$ and relation $g_0\wedge g_1=0$.
  Let $A_i=\gen{\{g_j\suchthat j\neq i\}}$. Then $(A_i)_{i<3}$ commutes but $(A_i)_{i<2}$ does not.
\end{example}
\begin{example}
  Let $B$ be freely generated by elements $g_0,g_1,g_2$ and relation $\bigwedge_{i<3}g_i=0$.
  Let $A_i=\gen{\{g_i\}}$. Then $(A_i,A_j)$ is independent for all $i<j<3$
  but $(A_i)_{i<3}$ does not commute.
\end{example}
\begin{example}
  Let $B$ be freely generated by elements $g_0,g_1$; let
  \[g_2=(g_0\wedge-g_1)\vee(-g_0\wedge g_1)\] and $A_i=\gen{\{g_i\}}$ for $i<3$.
  Then $(A_i,A_j)$ is independent for all $i<j<3$ and 
  $(A_i,\gen{A_j\cup A_k})$ commutes for all $i,j,k<3$.
  Yet, $(A_i)_{i<3}$ does not commute.  
\end{example}

To define the Stone dual of a commuting tuple of Boolean subalgebras,
we first abstract from the inclusion $\id\colon A_i\to B$
to a Boolean embedding $e_i\colon A_i\to B$
and abstract from the intersection $A_i\cap A_j$ to a limit
\(\xymatrix@C=1pc{A_i&&\ar[ll]_{e_{j,i}}A_{\{i,j\}}\ar[rr]^{e_{i,j}}&&A_j}\)
of the cospan \(\xymatrix@C=1pc{A_i\ar[r]^{e_i}&B&\ar[l]_{e_j}A_j}\).
Then it is clearer how to apply Stone duality.

\begin{defn}\label{multisurjdef}
  In the category of Stone spaces and continuous maps,
  a tuple of surjective morphisms $q_i\colon Y\to X_i$ for $i\in S$
  is \emph{multisurjective} if, for every (equivalently, some) choice of colimits
  \(\xymatrix@C=1pc{X_i\ar[rr]^{q_{j,i}}&&X_{\{i,j\}}&&\ar[ll]_{q_{i,j}} X_j}\)
  of the spans \(\xymatrix@C=1pc{X_i&\ar[l]_{q_i}Y\ar[r]^{q_j}&X_j}\)\!\!,
  and for every $\vec x\in\prod\vec X$ satisfying
  $q_{j,i}(x_i)=q_{i,j}(x_j)$ for all $i,j\in S$,
  there is some $y\in Y$ satisfying $q_i(y)=x_i$ for all $i\in S$.
\end{defn}
The above definition also makes sense for the category of
compact Hausdorff spaces and continuous maps.
In both categories, a colimit as in Definition~\ref{multisurjdef}
can be constructed using the finest topologically closed equivalence relation
on $Y$ containing $\{(a,b)\suchthat\exists k\in\{i,j\}\ \ q_k(a)=q_k(b)\}$.

For technical reasons, we will also need a weakening of
commuting Boolean subalgebra tuples.
For pairs, this property is exactly Heindorf and Shapiro's
definition of commuting pairs of Boolean subalgebras. \cite[Def. 4.1.1]{hs}
For longer tuples, it is a strictly weaker property.
\begin{defn}\label{weaklycommdef}
  Given a finite tuple $\vec A$ of subalgebras
  of a Boolean algebra $B$ and a subset $S_i$ of each $A_i$:
  \begin{itemize}
  \item Say that $\vec A$ \emph{weakly commutes at} $\vec S$ if,
    for each $\vec x\in\prod\vec S$ satisfying $\bigwedge\vec x=0$,
    there exists $\vec y\in\prod_i\gen{A_i\cap\bigcup_{j\neq i}A_j}$
    such that $\bigwedge\vec y=0$ and $x_i\leq y_i$ for all $i$.
  \item Say that $\vec A$ \emph{weakly commutes} if
    $\vec A$ weakly commutes at $\vec A$.
  \end{itemize}
\end{defn}

\begin{lem}\label{weaklybase}
  Suppose that $(A_i)_{i<n}$ weakly commutes at $(S_i)_{i<n}$
  and that each $S_i$ has $\vee$-closure $T_i$.
  Then $\vec A$ weakly commutes at $\vec T$.
\end{lem}
\begin{proof}
  Suppose that $\bigwedge_{i<n}x_i=0$ where $x_i=\bigvee_{k<m_i}s_i^k$
  where $m_i<\om$ and $s_i^k\in S_i$. Then, for each $f\in\prod\vec m$,
  we have $\bigwedge_{i<n}s_i^{f(i)}=0$ and, hence, some $(t_i^f)_{i<n}$
  such that $s_i^{f(i)}\leq t_i^f\in\gen{A_i\cap\bigcup_{j\neq i}A_j}$
  and $\bigwedge_it_i^f=0$.
  Letting $y_i=\bigvee_{k<m_i}\bigwedge_{f(i)=k}t_i^f$,
  we have $x_i\leq y_i\in\gen{A_i\cap\bigcup_{j\neq i}A_j}$ and
  \[\bigwedge\vec y=
  \bigwedge_{i<n}\bigvee_{k<m_i}\bigwedge_{f(i)=k}t_i^f
  =\bigvee_{g\in\prod\vec m}\bigwedge_{i<n}\bigwedge_{f(i)=g(i)}t_i^f
  \leq\bigvee_{g\in\prod\vec m}\bigwedge_{i<n}t_i^g=0.\qedhere\]
\end{proof}

\begin{lem}\label{weakercomm}
  If a finite tuple $(A_i)_{i<n}$ of subalgebras of a Boolean algebra $B$
  commutes, then it weakly commutes.
\end{lem}
\begin{proof}
  Let $\vec x\in\prod\vec A$ and $\bigwedge\vec x=0$.
  If $x_i=0$ for some $i$, then let $y_i=0$ and $y_j=1$ for $j\neq i$.
  Then $\vec y$ is as in Definition~\ref{weaklycommdef}.
  If $x_i>0$ for all $i$, then argue as follows.
  First, seeking a contradiction, suppose that for each finite subalgebra
  $F$ of $B$ containing $\{x_i\suchthat i<n\}$ there existed
  a compatible sequence of ultrafilters $\vec U_F\in\prod_i\ult(A_i\cap F)$,
  such that $\vec x\in\prod\vec U_F$. Then, choosing a fine ultrafilter $\mc{W}$
  supported on the set of all $F$, we declare
  $$U_i=\{b\in B\suchthat\{F\suchthat b\in U_{F,i}\}\in\mc{W}\}$$
  to obtain a compatible sequence $\vec U\in\prod_i\ult(A_i)$
  such that $\vec x\in\prod\vec U$. Because $\bigwedge\vec x=0$,
  this $\vec U$ contradicts our hypothesis that $\vec A$ is commuting.

  Thus, we may choose a finite subalgebra $F$ containing each $x_i$
  such that every sequence of ultrafilters $\vec U_F\in\prod_i\ult(A_i\cap F)$
  satisfying $\vec x\in\prod\vec U_F$ is incompatible.
  Let $S_i$ be the set of atoms of $A_i\cap F$ below $x_i$, for each $i$.
  Then, for each $\vec z\in\prod\vec S$, there exist $i<j<n$ and
  $c\in A_i\cap A_j\cap F$ such that $z_i\leq c$ and $z_j\leq -c$.
  Letting $w_i=c$, $w_j=-c$, and $w_k=1$ for $k\in n\setminus\{i,j\}$
  witnesses that $\vec A$ weakly commutes at $(\{z_i\})_{i<n}$. 
  By Lemma~\ref{weaklybase}, $\vec A$ also weakly commutes at $(\{x_i\})_{i<n}$.
\end{proof}

The converse of the above lemma is true for pairs of subalgebras
but not in general.

\begin{lem}\label{commpair}
  If a pair of subalgebras $A,B$ of a Boolean algebra $C$ weakly commutes,
  then the pair commutes.
\end{lem}
\begin{proof}
  Suppose that $U\in\ult(A)$, $V\in\ult(B)$, and $(A,B)$ weakly commutes.
  Assuming also that $U\cup V$ does not extend to an ultrafilter of $C$,
  \istst\ $U$ and $V$ are incompatible. By assumption, there exist
  $x\in U$ and $y\in V$ such that $x\perp y$. By weak commuting,
  there are $x',y'\in A\cap B$ such that $x\leq x'\perp y'\geq y$.
  Since $U$ and $V$ are upward closed, we have
  $x'\in U\cap B$ and $y'\in V\cap A$. Thus, $U$ and $V$ are incompatible.
\end{proof}
\begin{example}
  Let $B$ be freely generated by elements $g_0,g_1,g_2$
  and relation $\bigwedge_{i<3}g_i=0$.
  Let $A_i=\gen{\{g_j\suchthat j\neq i\}}$ for $i<3$.
  Then $\vec A$ weakly commutes but does not commute.
\end{example}

In the above counterexample, $(A_i)_{i<2}$ does not weakly commute.
Failure of an initial segment of a tuple to weakly commute is the
only obstruction to the whole tuple commuting.

\begin{lem}\label{commstepup}
  Suppose that a Boolean algebra $B$ has subalgebras $A_0,\ldots,A_n$ such that
  \begin{enumerate}
  \item $(A_i)_{i<n}$ commutes,
  \item $(\gen{\bigcup_{i<n}A_i},A_n)$ commutes, and
  \item $A_n\cap\gen{\bigcup_{i<n}A_i}=\gen{A_n\cap\bigcup_{i<n}A_i}$.
  \end{enumerate}
  Then $(A_i)_{i<n+1}$ commutes.
\end{lem}
\begin{proof}
  Suppose that $\vec U\in\prod_{i<n+1}\ult(A_i)$ is compatible.
  We will find $W\in\ult(B)$ extending $\bigcup\vec U$.
  By (1), $\bigcup_{j<n}U_j$ extends to some $V\in\ult(C)$
  where $C=\gen{\bigcup_{j<n}A_j}$. By (2), \istst\ $(V,U_n)$ is compatible.
  Suppose $x\in U_n\cap C$. Then $x\in D$ where $D=A_n\cap C$.
  By (3), $x\in\gen{\bigcup_{j<n}(A_n\cap A_j)}$;
  let $x=\bigvee_{i<m}\bigwedge_{j<n}y_{i,j}$
  for some $m$ and $y_{i,j}\in A_n\cap A_j$.
  Since $U_n\cap C\in\ult(D)$, we have $\bigwedge_{j<n}y_{i,j}\in U_n$ for some $i$.
  But then $y_{i,j}\in U_n\cap A_j\subset U_j\subset V$.
  Therefore, $x\in V$. Thus, $U_n\cap C\subset V$
  and, hence, $(V,U_n)$ is compatible.
\end{proof}

\begin{lem}\label{commcohere}
  Suppose that $(A_i)_{i<n+1}$ weakly commutes.
  Then $\left(\gen{\bigcup_{i<n}A_i},A_n\right)$ commutes and
  \[A_n\cap\gen{\bigcup_{i<n}A_i}=\gen{A_n\cap\bigcup_{i<n}A_i}.\]
\end{lem}
\begin{proof}
  Letting $B=\gen{\bigcup_{i<n}A_i}$, suppose that $A_n\ni x\perp y\in B$ and
  $y=\bigwedge_{i<n}y_i$ where $y_i\in A_i$.
  Let $z_i=y_i$ and $z_n=x$. Then there exists
  $\vec w\in\prod_{i<n+1}\overlap$
  such that $z_i\leq w_i$ and $\bigwedge\vec w=0$.
  Therefore, $x\leq w_n\perp y$ and $w_n\in C\subset A_n\cap B$
  where $C=\gen{A_n\cap\bigcup_{i<n}A_i}$
  Thus, $(A_n,B)$ weakly commutes at $(\{x\},\{y\})$.
  Every $b\in B$ is of the form $\bigvee_{i<m}\bigwedge_{j<n}b_{i,j}$
  where $b_{i,j}\in A_j$.
  Therefore, by Lemma~\ref{weaklybase}, $(A_n,B)$ weakly commutes and, 
  by Lemma~\ref{commpair}, commutes.

  It remains to show that $A_n\cap B\subset C$. Since every $b\in B$ is of
  the form $\bigvee_{i<m}\bigwedge_{j<n}b_{i,j}$ where $b_{i,j}\in A_j$, \istst\
  if $a_i\in A_i$ for $i<n+1$ and $\bigwedge_{i<n}a_i=a_n$ then $a_n\in C$.
  By our argument in the preceding paragraph, there exists $c\in C$ such
  that $\bigwedge_{i<n}a_i\leq c\perp-a_n$. But then $a_n=c\in C$.
\end{proof}

\begin{lem}\label{wellcomm}
  If $n<\om$ and $(A_i)_{i<m}$ weakly commutes for all $m\leq n$,
  then $(A_i)_{i<m}$ commutes for all $m\leq n$.
\end{lem}
\begin{proof}
  Case $n\leq 1$ is trivial.
  To induct from $n=k\geq 1$ to $n=k+1$, apply Lemma~\ref{commcohere}
  and then Lemma~\ref{commstepup}.
\end{proof}

Next we prove two easy lemmas about interactions between
commuting sequences and unions.

\begin{lem}\label{commgroup}
  If $(A_i)_{i<n}$ commutes and $\bigcup_{i<m}s(i)=n$, then
  $\left(\gen{\bigcup_{j\in s(i)}A_j}\right)_{i<m}$ commutes.
\end{lem}
\begin{proof}
  Letting $B_i=\gen{\bigcup_{j\in s(i)}A_j}$,
  suppose that $\vec U\in\prod_{i<m}\ult(B_i)$ and
  $U_i\cap B_j=U_j\cap B_i$ for $i,j<m$.
  Choose $r\colon n\to m$ such that $k\in s(r(k))$;
  let $V_k=U_{r(k)}\cap A_k$. We then have
  $V_k\cap A_l=V_l\cap A_k$ for $k,l<n$ and
  $U_i=B_i\vee\bigwedge_{k\in s(i)}V_k$ for $i<m$.
  Let $F$ be the possibly improper filter of $C=\gen{\bigcup\vec A}$
  generated by $\bigcup\vec V$, which is also the
  possibly improper filter generated by $\bigcup\vec U$.
  By hypothesis, $\bigcup\vec V$ and, hence, $F$
  extend to some $W\in\ult(C)$.
\end{proof}

\begin{lem}\label{commcts}
  Suppose that $n<\om$, $C$ is a Boolean algebra, and, for each $i<n$,
  \begin{itemize}
  \item $\mc{A}_i$ is a set of subalgebras of $C$,
  \item $\bigcup\mc{A}_i$ is a subalgebra of $C$, and
  \item $\vec A$ commutes for each $\vec A\in\prod\vec{\mc{A}}$.  
  \end{itemize}
  Then $\vec B$ commutes where $B_i=\bigcup\mc{A}_i$.
\end{lem}
\begin{proof}
  Given $\vec U\in\prod_{i<n}\ult(B_i)$ whose union does not extend to any
  $V\in\ult(C)$, \istst\ $U_i\cap B_j\not=U_j\cap B_i$ for some $i,j<n$.
  By hypothesis, there exists $\vec x\in\prod\vec U$ such that $\bigwedge\vec x=0$.
  Choose $\vec A\in\prod\vec{\mc{A}}$ such that $\vec x\in\prod\vec A$.
  Then $\bigcup_{i<n}(U_i\cap A_i)$ does not extend to any $V\in\ult(C)$. 
  Hence, $U_i\cap A_i\cap A_j\not= U_j\cap A_j\cap A_i$ for some $i,j<n$.
\end{proof}

\subsection{The other half of the main theorem}
\begin{defn}\label{dksfndef}
  Given $A$ a Boolean algebra and $d<\om\leq\ka=\cf(\ka)$:
  \begin{itemize}
  \item 
    We say that $A$ has the \emph{$d$-ary strong $\ka$-Freese-Nation property},
    or \dksfn{d}{\ka}\ for short, if there exists a club\footnote{
      Cofinal and closed with respect to unions of increasing sequences
      of length less than $\ka$.} subset $\mc{C}$ of $[A]^{<\ka}$
    such that every $\vec{B}\in\mc{C}^d$ is a commuting tuple of subalgebras of $A$.
  \item We say that $A$ has the \dksfn{<\om}{\ka}\ if
    there exists a common $\mc{C}$ witnessing the \dksfn{d}{\ka} for all $d$.
  \end{itemize}
\end{defn}
The \dksfn{2}{\alo}\ is exactly the \emph{strong Freese-Nation property} (SFN)
introduced by Heindorf and Shapiro in \cite{hs}. Heindorf and Shapiro showed
that for Boolean algebras of size at most $\aleph_1$,
the FN, the SFN, and projectivity are all equivalent.
In \cite{msfn}, the author constructed a Boolean algebra
of size $\al_2$ with the FN but without the SFN.
To connect the \dksfn{d}{\ka} with the existence of Geschke maps,
we will also use the following variant of the Freese-Nation property.

\begin{defn}\label{mnksfndef}
  Given $A$ a Boolean algebra, ordinals $\zeta,\eta$,
  and a regular infinite cardinal $\ka$, we say that
  $A$ has the \emph{$(\zeta,\eta)$-strong $\ka$-Freese-Nation property},
  or \dksfn{\zeta,\eta}{\ka}\ for short, if there exists $\mc{C}\subset[A]^{<\ka}$
  such that $A=\bigcup\mc{C}$, every $B\in\mc{C}$ is a subalgebra of $A$,
  and $\left(\gen{\bigcup\vec{H}},\gen{\bigcup\vec{K}}\right)$ commutes
  for every $\vec{H}\in\mc{C}^\zeta$ and $\vec{K}\in\mc{C}^\eta$.
\end{defn}

\begin{lem}\label{sfngrouping}
  Let $\ka$ be a regular infinite cardinal.
  Given finite ordinals $m,n$ and a Boolean algebra with the
  \dksfn{m+n}{\ka}, the algebra also has the \dksfn{m,n}{\ka}.
  Given arbitrary ordinals $\zeta,\eta$ and a Boolean algebra with the
  \dksfn{<\om}{\ka}, the algebra also has the \dksfn{\zeta,\eta}{\ka}.
\end{lem}
\begin{proof}
  The first claim immediately follows from Lemma~\ref{commgroup}.
  For the second claim, Lemma~\ref{commgroup} yields the \dksfn{\zeta,\eta}{\ka}
  for all $\zeta,\eta<\om$; Lemma~\ref{commcts} improves this
  to \dksfn{\zeta,\eta}{\ka} for all ordinals $\zeta,\eta$.
\end{proof}

\begin{lem}\label{sfntogeschke}
  If $\tau$ is a nonzero cardinal and $A$ has the \dksfn{1,\tau}{\ka},
  then $A$ has a $(\ka,\tau)$-Geschke map.
\end{lem}
\begin{proof}
  Let $\mc{C}$ witness the \dksfn{1,\tau}{\ka}\ as in Definition~\ref{mnksfndef}.
  For each $p\in A$, let $f(p)$ be some $C\in\mc{C}$ containing $p$,
  noting that this implies $\card{f(p)}<\ka$.
  Now suppose that $a\in A$ and $B=\gen{\bigcup_{\af<\tau}E_\af}$ where
  each $E_\af$ is an $f$-closed subalgebra of $A$.
  \Istst\ $\cf(I)<\ka$ where $I=B\cap[0,a]$.
  Therefore, given $b\in I$, \istf\ $z\in I\cap f(-a)$ such that $b\leq z$.
  The set $I\cap f(-a)$ is $\vee$-closed because $f(-a)$ is a subalgebra of $A$.
  Therefore, \wma\ that $b=\bigwedge\vec e$ for some finite $\sigma\subset\tau$
  and $\vec e\in\prod_{\af\in\sigma}E_\af$.
  We then have $x_0\perp x_1$ where $x_0=\bigwedge_{\af\in\sigma}e_\af$ and $x_1=-a$.
  Letting $C_0=\gen{\bigcup_{\af\in\sigma}f(e_\af)}$ and $C_1=f(-a)$, the pair $(C_0,C_1)$
  commutes and, by Lemma~\ref{weakercomm}, weakly commutes.
  Therefore, there exist $y_0,y_1\in C_0\cap C_1$ such that
  $x_0\leq y_0\perp y_1\geq x_1$.
  Then $z=y_0$ is as desired: $z\in f(-a)$, $b=\bigwedge\vec e\leq z\leq a$,
  and $z\in\gen{\bigcup_{\af\in\sigma}f(e_\af)}\subset\gen{\bigcup\vec E}\subset B$.
\end{proof}

\begin{lem}\label{weakcommgrow}
  Let $\vn\neq\sigma\subset n<\om$ and let $D$ be a Boolean algebra with
  element $c$ and subalgebras $A_i$ for $i<n$ such that $(A_i)_{i<n}$
  weakly commutes and, for each $d\in\{\pm c\}$, the set
  $[0,d]\cap\bigcap_{i\in\sigma}A_i$ is cofinal
  in $[0,d]\cap\gen{\bigcup_{i<n}A_i}$. Then $(B_i)_{i<n}$ weakly commutes
  where $B_i$ is $\gen{A_i\cup\{c\}}$ if $i\in\sigma$ else $A_i$.
\end{lem}
\begin{proof}
  For each $i<n$, suppose that we have $b_i=a_i\wedge c_i$ where  $a_i\in A_i$,
  $c_i\in\{\pm c\}$ if $i\in\sigma$, and $c_i=1$ if $i\not\in\sigma$.
  By Lemma~\ref{weaklybase}, \istst\ $\vec B$ weakly commutes
  at $(\{b_i\})_{i<n}$. So, suppose $\bigwedge_{i<n}b_i=0$. We must find
  $\vec y\in\prod_{i<n}\gen{B_i\cap\bigcup_{j\neq i}B_j}$ such that
  $\vec y\geq\vec b$ and $\bigwedge\vec y=0$.
  If $c_i=c$ and $c_j=-c$ for some $i,j<n$, then we obtain $\vec y$ as desired
  by setting $y_i=c$, $y_j=-c$, and $y_k=1$ for $k\not\in\{i,j\}$.
  Therefore, \wma\ that $c_i=c$ for all $i\in\sigma$.
  Then $\bigwedge_{i<n}a_i\leq -c$ and, hence, there exists
  $e\in\bigcap_{i\in\sigma}A_i$ such that $\bigwedge_{i<n}a_i\leq e\leq -c$.
  Let $x_i=a_i\wedge-e$ for $i\in\sigma$ and $x_i=a_i$ for $i\not\in\sigma$.
  Then $\vec b\leq\vec x\in\prod\vec A$ and $\bigwedge\vec x=0$.
  Therefore, there exists $\vec y\in\prod_i\gen{A_i\cap\bigcup_{j\neq i}A_j}$
  such that $\vec x\leq\vec y$ and $\bigwedge\vec y=0$. Thus, we have found
  $\vec y$ witnessing that $\vec B$ weakly commutes at $(\{b_i\})_{i<n}$.
\end{proof}

\begin{lem}\label{tighttosfn}
  If $A$ is tightly $\ka$-filtered, then $A$ has the \dksfn{<\om}{\ka}.
\end{lem}
\begin{proof}
  We modify Geschke's proof of (ii)$\Rightarrow$(iii) of
  \cite[Thm. 2.5]{geschke}.
  Let $A$ be tightly $\ka$-filtered. Specifically, let
  $A$, $\ka$, and $(x_\af)_{\af<\eta}$ be as in
  Definition~\ref{deftight}. For each $T\subset\eta$, let
  $A(T)=\gen{\{x_\af\suchthat\af\in T\}}$. For each $\af<\eta$, choose
  $f(\af)\in[\af]^{<\ka}$ such that $A(f(\af))$ contains a cofinal
  subset of $[0,x_\af]\cap A(\af)$ and a cofinal subset of
  $[0,-x_\af]\cap A(\af)$. For each $S\in[\eta]^{<\ka}$, let $F(S)$
  be the $f$-closure of $S$. By induction on
  $\sup_{\af\in S}(\af+1)$, we have $\card{F(S)}<\ka$:
  if $S$ has maximum $\af$, then \[F(S)=F(S\cap\af)\cup F(f(\af))\cup\{\af\};\]
  if $S$ has a cofinal sequence $(\zeta_\bt)_{\bt<\mu}$ where $\mu$ is
  regular and $\om\leq\mu<\ka$, then $F(S)=\bigcup_\bt F(S\cap\zeta_\bt)$.
  Let $\mc{E}=\{A(F(S))\suchthat S\in[\eta]^{<\ka}\}$, which is club
  in $[A]^{<\ka}$. By Lemma~\ref{wellcomm}, \istst\ every tuple in
  $\mc{E}^{<\om}$ weakly commutes.

  Let $n<\om$ and $S_i\in[\eta]^{<\ka}$ for $i<n$. By induction on $\af$,
  we will prove that $(A(F(S_i)\cap\af))_{i<n}$ weakly commutes for all
  $\af\leq\eta$. The limit case is automatic. For the successor case,
  let $B_i=A(F(S_i)\cap\af)$ and $C_i=A(F(S_i)\cap(\af+1))$, and
  suppose that $(B_i)_{i<n}$ weakly commutes and that $\sigma\neq\vn$
  where $\sigma=\{i<n\suchthat\af\in F(S_i)\}$. Then $\bigcap_{i\in\sigma}B_i$
  contains $A(f(x_\af))$ and $A(f(x_\af))$ contains, for each $y\in\{\pm x_\af\}$,
  a cofinal subset of $[0,y]\cap A(\af)$. Since
  $\bigcap_{i\in\sigma}B_i\subset\gen{\bigcup_{i<n}B_i}\subset A(\af)$,
  the set $[0,y]\cap\bigcap_{i\in\sigma}B_i$ is cofinal in
  $[0,y]\cap\gen{\bigcup_{i<n}B_i}$. By Lemma~\ref{weakcommgrow},
  $(C_i)_{i<n}$ weakly commutes.
\end{proof}

\begin{thm}\label{tightsfnmap}
  Given a regular infinite cardinal $\ka$ and a Boolean algebra $A$, \tfae.
  \begin{enumerate}
  \item $A$ is tightly $\ka$-filtered.
  \item $A$ has the \dksfn{<\om}{\ka}.
  \item $A$ has a $(\ka,\tau)$-Geschke map for all nonzero cardinals $\tau$.
  \item $A$ has the \dksfn{d+1}{\ka}\ for all $d\in[1,\om)$ satisfying
    $\ka^{+d}\leq\card{A}$.
  \item $A$ has the \dksfn{1,d}{\ka}\ for all $d\in[1,\om)$ satisfying
    $\ka^{+d}\leq\card{A}$.
  \item $A$ has a $(\ka,d)$-Geschke map for all $d\in[1,\om)$ satisfying
    $\ka^{+d}\leq\card{A}$.
  \end{enumerate}
\end{thm}
\begin{proof}\
  \begin{enumerate}
  \item[(1)$\Rightarrow$(2)] Lemma~\ref{tighttosfn}.
  \item[(2)$\Rightarrow$(4)] Trivial.
  \item[(4)$\Rightarrow$(5)] Lemma~\ref{sfngrouping}.
  \item[(5)$\Rightarrow$(6)] Lemma~\ref{sfntogeschke}.
  \item[(2)$\Rightarrow$(3)]
    Lemma~\ref{sfngrouping} followed by Lemma~\ref{sfntogeschke}.
  \item[(3)$\Rightarrow$(6)] Trivial.
  \item[(6)$\Rightarrow$(1)] By Lemma~\ref{geschketogame},
    II has an uncoordinated winning strategy for $\mc{G}(A,\ka,d)$
    for all $d\in[1,\om)$ satisfying $\ka^{+d}\leq\card{A}$.
    Therefore, by Theorem~\ref{tightgameclub}, $A$ is tightly $\ka$-filtered.
  \end{enumerate}
\end{proof}

\begin{cor}\label{projlocfin}
  Given a Boolean algebra $A$, \tfae.
  \begin{enumerate}
  \item $A$ is projective.
  \item $A$ has a cofinal family of finite subalgebras, 
    any finite tuple of which commutes.
  \item $A$ has a cofinal family of finite subalgebras, 
    any finite tuple of which weakly commutes.
  \item $A$ has a cofinal family of finite subalgebras, 
    any $n$-tuple of which commutes
    if $2\leq n<\om$ and $\al_{n-1}\leq\card{A}$.
  \item $A$ has a cofinal family of finite subalgebras, 
    any $n$-tuple of which weakly commutes
    if $2\leq n<\om$ and $\al_{n-1}\leq\card{A}$.\qedhere
  \end{enumerate}  
\end{cor}

\begin{cor}\label{aezerolocfin}
  A topological space $X$ is Dugundji iff $X$ is an inverse limit of maps
  $(q_{\sigma,\tau}\colon m_\tau\to m_\sigma)_{\sigma\trile\tau\in D}$ where:
  \begin{itemize}
  \item $(D,\trile)$ is an upward directed poset,
  \item $m_\sigma$ is a finite discrete space for each $\sigma\in D$,
  \item $q_{\rho,\sigma}\circ q_{\sigma,\tau}=q_{\rho,\tau}$
    for each triple $\rho\trile\sigma\trile\tau$, and
  \item $(q_{\sigma_i,\tau})_{i<n}$ is multisurjective for all $n<\om$,
  $\tau\in D$, and $\sigma_0,\ldots,\sigma_{n-1}\trile\tau$.
  \end{itemize}
  The same is true if we require $n$ to also satisfy $\al_n\leq\weight{X}^+$.
\end{cor}

\begin{thm}[Main Theorem repeated]\label{mainthmrepeat}
Given a regular infinite cardinal $\ka$ and a Boolean algebra $A$, \tfae.
\begin{enumerate}
\item $A$ is tightly $\ka$-filtered.
\item For every nonzero cardinal $\tau$, $A$ has a $(\ka,\tau)$-Geschke map.
\item For every finite nonzero ordinal $\tau$ satisfying $\ka^{+\tau}\leq\card{A}$, 
$A$ has a $(\ka,\tau)$-Geschke map.
\item For every nonzero cardinal $\tau$, 
II has an uncoordinated winning strategy for $\mc{G}(A,\ka,\tau)$.
\item For every finite nonzero ordinal $\tau$ satisfying $\ka^{+\tau}\leq\card{A}$, 
II has an uncoordinated winning strategy for $\mc{G}(A,\ka,\tau)$.
\end{enumerate}
\end{thm}
\begin{proof}\
  \begin{enumerate}
  \item[(1)$\Rightarrow$(2)] Theorem~\ref{tightsfnmap}.
  \item[(2)$\Rightarrow$(3)] Trivial.
  \item[(3)$\Rightarrow$(5)] Lemma~\ref{geschketogame}.
  \item[(2)$\Rightarrow$(4)] Lemma~\ref{geschketogame}.
  \item[(4)$\Rightarrow$(5)] Trivial.
  \item[(5)$\Rightarrow$(1)] Theorem~\ref{tightgameclub}.
  \end{enumerate}
\end{proof}

\subsection{The other half of the main example}
\begin{lem}\label{hypergraphgood}
  Given $1\leq d<n<\om\leq\lm$, the algebra $C_{\lm,d,n}$ has the \dksfn{d}{\alo}.
\end{lem}
\begin{proof}
  Let $\mc{E}$ be the set of all subalgebras of $C_{\lm,d,n}$ of
  the form $E(\sigma)=\gen{\{f(\tau)\suchthat\tau\subset\sigma\}}$ where
  $\sigma\in[\lm]^{<\alo}$ and $f$ is as in Definition~\ref{hypergraphdef}.
  Then $\mc{E}$ is cofinal in $[C_{\lm,d,n}]^{<\alo}$.
  Given $\vec\sigma\in\left([\lm]^{<\alo}\right)^d$, we will show
  that $(E(\sigma_i))_{i<d}$ commutes. Seeking a contradiction, suppose that
  $\vec U\in\prod_i\ult(E(\sigma_i))$ is a sequence of compatible
  ultrafilters whose union does not extend to an ultrafilter of $C_{\lm,d,n}$.
  Letting $\Gamma_i=\{\tau\suchthat f(\tau)\in U_i\}$, we have
  $[\upsilon]^d\not\subset\Gamma_i$ for all $\upsilon\in[\sigma_i]^n$
  because each $U_i$ is a proper filter, we have
  $\Gamma_i\cap[\sigma_j]^d=\Gamma_j\cap[\sigma_i]^d$
  because $\vec U$ is compatible, and we have
  $[\rho]^d\subset\bigcup\vec\Gamma$ for some $\rho\in[\bigcup\vec\sigma]^n$
  because $\bigcup\vec U$ does not generate a proper filter.
  But then $[\rho]^d\subset\bigcup_i[\sigma_i]^d$ and,
  by Lemma~\ref{trivialcliquecover}, $\rho\subset\sigma_k$ for some $k$.
  Thus, $[\rho]^d\subset\Gamma_k$ and we have a contradiction.
\end{proof}
Note that the above lemma is trivial for $d=1$ because every Boolean
algebra trivially has the \dksfn{1}{\alo}.

\begin{thm}\label{hypergraphsame}
  Given $2\leq d<n<\om\leq\lm<\om_d$, the space $X_{\lm,d,n}$
  is homeomorphic to $2^\lm$.
\end{thm}
\begin{proof}
  By Lemma~\ref{hypergraphgood}, $C_{\lm,d,n}$ has \dksfn{d}{\alo}.
  By Corollary~\ref{projlocfin}, $C_{\lm,d,n}$ is projective
  and $X_{\lm,d,n}$ is Dugundji.
  Also, every point in $X_{\lm,d,n}$ is easily verified to
  have character $\lm$. (This is not true if we allow $d=1$.)
  Therefore, by \scepin's Theorem \cite[Thm. 9]{sctoplim},
  $C_{\lm,d,n}$ is free and $X_{\lm,d,n}\homeo 2^\lm$.\footnote{
  See \cite[Thm. 1.4.11]{hs} for an algebraic statement and proof
  of \scepin's Theorem. An essential tool is Sirota's lemma \cite{sirota},
  which says that every continuous open surjection $p\colon2^\om\to2^\om$
  with perfect fibers is isomorphic to the first coordinate projection
  from $(2^\om)^2$ to $2^\om$.}
\end{proof}

\begin{thm}\label{hypergraphdifference}
  Given $1\leq e<d<m<\om\leq\cf(\ka)=\ka<\ka^{+e}\leq\lm$ and $e<n<\om$,
  II has an uncoordinated winning strategy for $\mc{G}(C_{\lm,d,m},\ka,e)$
  but not for $\mc{G}(C_{\lm,e,n},\ka,e)$.
  In particular, the spaces $X_{\lm,d,m}$ and $X_{\lm,e,n}$ are not homeomorphic.
\end{thm}
\begin{proof}
  By Lemma~\ref{hypergraphgood}, $C_{\lm,d,m}$ has the \dksfn{d}{\alo}.
  By Lemma~\ref{sfngrouping}, $C_{\lm,d,m}$ has the \dksfn{d-1,1}{\alo}.
  By Lemma~\ref{sfntogeschke}, $C_{\lm,d,m}$ has an $(\alo,d-1)$-Geschke map,
  which is also a $(\ka,e)$-Geschke map. By Lemma~\ref{geschketogame},
  II has an uncoordinated winning strategy for $\mc{G}(C_{\lm,d,m},\ka,e)$.
  By Theorem~\ref{hypergraphbad},
  every club $\mc{F}\subset[C_{\lm,e,n}]^{\leq\ka}$ contains $\mc{B}\in[\mc{F}]^e$
  such that $\gen{\bigcup\mc{B}}\not\leq_\ka C_{\lm,e,n}$.
  By Lemma~\ref{geschketogame}, II does not have 
  an uncoordinated winning strategy for $\mc{G}(C_{\lm,e,n},\ka,e)$.
\end{proof}

For $2\leq d<e$,
I am not aware of another way to topologically distinguish
the spaces $X_{\lm,d,m}$ and $X_{\lm,e,n}$ in the above theorem.
In particular, a simple $\Delta$-system argument shows that
every regular uncountable cardinal is a caliber\footnote{
  An ordinal $\eta$ is a caliber of a space if
  for every $\eta$-sequence of nonempty open sets
  there is a point common to all sets in some $\eta$-long subsequence.}
of $X_{\mu,p,q}$ if $2\leq p<q<\om\leq\mu$.
On other hand, if $1<h<\om\leq\mu$, then $X_{\mu,1,h}$ is a scattered
space with height $h$ and cellularity $\mu$.

\begin{thm}[Main Example repeated]
Given a regular infinite cardinal $\ka$ and $0<d<\om$,
let $C_{\ka,d}$ be the clopen algebra of the space of $d$-uniform hypergraphs
on $\ka^{+d}$ that avoid copies of $[d+1]^d$.
Then $C_{\ka,d}$ has an $(\alo,d-1)$-Geschke map but II does not have 
an uncoordinated winning strategy for $\mc{G}(C_{\ka,d},\ka,d)$.
\end{thm}
\begin{proof}
  By Theorem~\ref{hypergraphdifference}, II does not have 
  an uncoordinated winning strategy for $\mc{G}(C_{\ka^{+d},d,d+1},\ka,d)$.
  The proof of that theorem also shows that $C_{\ka^{+d},d,d+1}$ has an
  $(\alo,d-1)$-Geschke map.
\end{proof}

\subsection{Purely finite consequences}\label{purefin}
In this subsection we observe strictly finitary consequences
of the preceding subsections.
First, we observe that propositional logic satisfies
a higher-arity version of the Craig interpolation property.

\begin{thm}
  Suppose that $n<\om$ and $\bigwedge_{i<n}\varphi_i\proves\bot$
  where each $\varphi_i$ is a formula of propositional logic.
  Then there exists formulas $\psi_i$ for $i<n$ such that
  $\varphi_i\proves\psi_i$, $\bigwedge_{i<n}\psi_i\proves\bot$, 
  and, for each propositional variable $v$ of $\psi_i$,
  there exists $j\neq i$ such that $v$ is common to $\varphi_i$ and $\varphi_j$.
\end{thm}
\begin{proof}
  For each $i$, let $S_i$ be the set of propositional variables of $\varphi_i$.
  For each propositional variable $v\in\bigcup\vec S$,
  let $A_v$ be the Boolean algebra $\{\top,\bot,v,\neg v\}$.
  The sequence $(A_v\suchthat v\in\bigcup\vec S)$ is independent.
  Therefore, by Lemma~\ref{commgroup}, $(\gen{\bigcup_{v\in S_i}A_v})_{i<n}$
  is commuting and, by Lemma~\ref{weakercomm}, weakly commuting.
  Thus, there exists $\vec\psi$ as desired.
\end{proof}

Second, we observe that the nontrivial symmetric power functors $\spowop^n$
for $n\geq 2$ and the Vietoris hyperspace functor $\Exp$ fail
to preserve multisurjectivity of triples of morphisms in the category
of finite sets and finite functions.
Though the proof given here involves $2^{\om_2}$,
a strictly finitary proof is presumably feasible.
However, for me to even conjecture the following theorem
required the motivation of the preceding subsections.

\begin{defn}\
  In the category $\mb{K}$ of compact Hausdorff spaces and continuous maps,
  and also in the full subcategories
  $\mb{K}_0$ of Stone spaces and continuous maps
  and $\mb{F}$ of finite spaces with finite functions,
  define the following endofunctors.
  \begin{itemize}
  \item $\spowop^n(X)$ is $X^n/{\sim}$ where $p\sim q$ iff $p=q\circ\sigma$
    for some permutation of $n$ (and $n$ is finite).
  \item Given $f\colon X\to Y$, define
    $\spowop^n(f)\colon\spowop^n(X)\to\spowop^n(Y)$ by
    \[f((p_i)_{i<n}/{\sim})=(f(p_i))_{i<n}/{\sim}.\]
  \item $\Exp(X)$ is the space of all nonempty closed subsets of $X$
    with the topology generated by sets of the form
    $\{K\suchthat K\cap U\neq\vn\}$ or
    $\{K\suchthat K\subset U\}$ for $U$ open in $X$.
  \item Given $f\colon X\to Y$,
    define $\Exp(f)\colon\Exp(X)\to\Exp(Y)$ by \[\Exp(f)(K)=f[K].\]
  \end{itemize}
\end{defn}

\begin{thm}\label{badmultisurjfunctor}
  If $\Phi$ is $\Exp$ or $\spowop^n$ for some $n\geq 2$, then there
  exists a triple of finite functions $(h_i)_{i<3}$ such that
  $(h_i)_{i\in\sigma}$ is multisurjective for all $\sigma\subset 3$
  but $(\Phi(h_i))_{i<3}$ is not multisurjective.
\end{thm}
\begin{proof}
  There is a sequence of finite functions
  $(q_{\sigma,\tau}\colon m_\tau\to m_\sigma)_{\sigma\trile\tau\in D}$
  as in Corollary~\ref{aezerolocfin} with inverse limit $2^{\om_2}$.
  (Indeed, we can take $D=[\om_2]^{<\alo}$, ${\trile}={\subset}$,
  $m_\sigma=2^\sigma$, and $q_{\sigma,\tau}(x)=x\restrict\sigma$.)
  In particular, any finite tuple of these maps of the form
  $(q_{\sigma_i,\tau})_{i<m}$ is multisurjective.
  If $\Phi$ preserved multisurjectivity of triples of finite functions then,
  applying Corollary~\ref{aezerolocfin} again,
  $\Phi(2^{\om_2})$ would be Dugundji.
  But \scepin\ proved that $\Phi(2^{\om_2})$ is not Dugundji.~\cite{scfunct}
  (Shapiro had earlier proved that $\Exp(2^{\om_2})$
  is not even dyadic.~\cite{shapiro}
  See also \cite[Ch. 3]{hs} for an algebraic statement
  and proof of the fact that $\Exp(2^{\om_2})$ and $\spowop^2(2^{\om_2})$
  are not Dugundji.)
\end{proof}
For $\Phi=\spowop^2$, Edward Estrada verified that the
$(h_i)_{i<3}$ of Theorem~\ref{badmultisurjfunctor}
can be extremely simple, with $h_i\colon2^3\to2^{3\setminus\{i\}}$
and $h_i(p)=(p(j))_{j\neq i}$. \cite{estrada}

\begin{cor}\label{finboolfunctor}
  If $\Phi$ is $\Exp$ or $\spowop^n$ for some $n\geq 2$
  in the category $\mb{K}_0$, and $\Psi$ is the Stone dual of $\Phi$,
  then there is finite Boolean algebra $B$ with a triple of
  commuting subalgebras $(A_i)_{i<3}$ such that the algebras
  $(\Psi(A_i))_{i<3}$ do not weakly commute when naturally identified
  with subalgebras of $\Psi(B)$.
\end{cor}
\begin{proof}
  Applying Stone duality to Theorem~\ref{badmultisurjfunctor},
  we obtain a finite Boolean algebra $B$ and subalgebras $(A_i)_{i<3}$
  such that $(A_i)_{i<2}$ commutes and $(A_i)_{i<3}$ commutes
  but $(\Psi(A_i))_{i<3}$ does not commute. Therefore,
  by Lemma~\ref{wellcomm}, $(\Psi(A_i))_{i<2}$ does not weakly commute  
  or $(\Psi(A_i))_{i<3}$ does not weakly commute.
  The functor $\Psi$ preserves commuting pairs of subalgebras.
  This is proved topologically in \cite{scfunct}.
  (An algebraic proof for $\Exp$ and $\spowop^2$ is given
  in \cite[Ch. 3]{hs}; the argument for $\spowop^2$
  is Observation (2) of Section 3.4, which easily generalizes to $\spowop^n$.)
  Therefore, $(\Psi(A_i))_{i<3}$ does not weakly commute.
\end{proof}
For $\Phi=\Exp$, Ren\'e Montemayor found that the finite Boolean algebra
with four atoms is the smallest $B$ for which some $\vec A$
witnesses Corollary~\ref{finboolfunctor}. \cite{montemayor}

\section{Questions}\label{secquestions}

\begin{question}
  Is the space of triangle-free graphs on $\om_2$ homeomorphic
  to the space of $K_4$-free graphs on $\om_2$?
  That is, is $X_{\om_2,2,3}\homeo X_{\om_2,2,4}$?
\end{question}

\begin{question}
  More generally, given $d,d'<\om$ and finite families
  $\mc{F},\mc{F}'\in[[\om]^{d}]^{<\alo}$ of forbidden induced subhypergraphs,
  what is the least infinite cardinal $\ka$ (if any) for which
  the space of all $\mc{F}$-free $d$-regular hypergraphs on $\ka$
  is not homeomorphic to
  the space of all $\mc{F}'$-free $d'$-regular hypergraphs on $\ka$?
\end{question}

\begin{question}
  Is $X_{\om_3,3,4}$ dyadic? More generally,
  is there a Boolean algebra $A$ such that
  II has an uncoordinated winning strategy for $\mc{G}(A,\alo,2)$
  but $A$ does not extend to a free Boolean algebra?
  Any such $A$ must have size at least $\al_3$.
\end{question}

\begin{question}
  Say that a Boolean algebra $A$ has the \dkfn{d}{\ka}
  if II has an uncoordinated winning strategy for $\mc{G}(A,\ka,\tau)$
  for all $\tau<d$. Does the \dkfn{3}{\alo} imply the \dksfn{3}{\alo}?
  Any counterexample must have size at least $\al_3$.
  (In~\cite{msfn} it was shown that
  \dkfn{2}{\alo}$\not\Rightarrow$\dksfn{2}{\alo}, and
  this paper's results immediately imply that
  \dksfn{d}{\ka}$\Rightarrow$\dkfn{d}{\ka},
  \dkfn{\om}{\ka}$\Leftrightarrow$\dksfn{<\om}{\ka}, and
  \dksfn{d}{\alo}$\not\Rightarrow$\dkfn{d+1}{\ka}.)
\end{question}

\end{document}